\documentclass[hidelinks,onefignum,onetabnum]{siamart220329}

\usepackage{lipsum}
\usepackage{amsfonts}
\usepackage{graphicx}
\usepackage{caption}
\usepackage{subcaption}
\usepackage{epstopdf}
\usepackage{algorithmic}
\usepackage{tikz,pgfplots}
\usepackage{pgfplots}
\usetikzlibrary{shapes.geometric, decorations.pathmorphing, positioning, arrows, arrows.meta, patterns, shapes, calc, fit, backgrounds}
\ifpdf
  \DeclareGraphicsExtensions{.eps,.pdf,.png,.jpg}
\else
  \DeclareGraphicsExtensions{.eps}
\fi
\graphicspath{{figures/SINUM_MFIE/}}

\newcommand{\R}{\mathbb R}
\newcommand{\C}{\mathbb C}

\newcommand{\EE}{\mathcal E}

\newcommand{\TT}{\mathcal T}
\newcommand{\OO}{\mathcal O}

\DeclareMathOperator{\Ls}{L}
\DeclareMathOperator{\LLs}{\mathbf L}

\DeclareMathOperator{\Hs}{H}
\DeclareMathOperator{\HHs}{\mathbf H}

\DeclareMathOperator{\WWs}{\mathbf W}
\DeclareMathOperator{\AAs}{\mathbf A}
\DeclareMathOperator{\DDs}{\mathbf D}
\DeclareMathOperator{\KKs}{\mathbf K}

\DeclareMathOperator{\CCs}{\mathbf C}
\DeclareMathOperator{\BBs}{\mathbf B}

\DeclareMathOperator{\UUs}{\mathbf U}
\DeclareMathOperator{\VVs}{\mathbf V}
\DeclareMathOperator{\SSs}{\mathbf S}
\DeclareMathOperator{\XXs}{\mathbf X}

\DeclareMathOperator{\cond}{cond}

\DeclareMathOperator{\RT}{RT}

\DeclareMathOperator{\MMs}{\mathbf M}
\DeclareMathOperator{\GGs}{\mathbf G}

\newcommand{\brac}[1]{\left\lbrace{#1}\right\rbrace}

\newcommand{\paren}[1]{\left({#1}\right)}
\newcommand{\norm}[1]{\left\lVert{#1}\right\rVert}

\newcommand{\abs}[1]{\left\vert{#1}\right\vert}

\newcommand{\inprod}[1]{\left\langle{#1}\right\rangle}

\newcommand{\wt}[1]{\widetilde{#1}}
\newcommand{\wh}[1]{\widehat{#1}}
\newcommand{\ovl}[1]{\overline{#1}}

\newcommand{\transpose}{\mathsf{T}}

\newcommand{\bm}[1]{\boldsymbol{#1}}

\newcommand{\curlt}{\textbf{\textup{curl}}}
\newcommand{\divt}{\textup{div}}
\newcommand{\gradt}{\textbf{grad}}

\newcommand{\veps}{\varepsilon}

\newcommand{\pa}{\partial}
\newcommand{\gm}{\gamma}
\newcommand{\Gm}{\Gamma}
\newcommand{\om}{\omega}
\newcommand{\Om}{\Omega}
\newcommand{\sm}{\sigma}

\newcommand{\vphi}{\varphi}
\newcommand{\fa}{\forall}
\newcommand{\sst}{\subset}

\newcommand{\Ab}{\bm{A}}

\newcommand{\xb}{\bm{x}}
\newcommand{\yb}{\bm{y}}
\newcommand{\zb}{\bm{z}}

\newcommand{\ab}{\bm{a}}
\newcommand{\eb}{\bm{e}}

\newcommand{\vb}{\bm{v}}
\newcommand{\ub}{\bm{u}}
\newcommand{\wb}{\bm{w}}

\newcommand{\fb}{\bm{f}}

\newcommand{\tv}{\textbf{t}}
\newcommand{\bb}{\bm{b}}

\newcommand{\nv}{\textbf{n}}

\newcommand{\zrb}{\bm{0}}
\newcommand{\vphib}{\bm{\varphi}}

\newcommand{\psib}{\bm{\psi}}
\newcommand{\xib}{\bm{\xi}}

\newcommand{\Psib}{\bm{\Psi}}

\newcommand{\dx}{\,\mathrm{d}\xb}

\newcommand{\ds}{\,\mathrm{d}s}

\newcommand{\q}{\quad}
\newcommand{\qq}{\qquad}
\newcommand{\qqq}{\qquad\quad}
\newcommand{\qqqq}{\qquad\qquad}
\newcommand{\qqqqq}{\qquad\qquad\quad}
\newcommand{\qqqqqq}{\qquad\qquad\qquad}

\newsiamremark{remark}{Remark}
\newsiamremark{hypothesis}{Hypothesis}
\crefname{hypothesis}{Hypothesis}{Hypotheses}
\newsiamthm{claim}{Claim}
\newsiamthm{assumption}{Assumption}

\headers{Boundary element methods for the MFIE on polyhedra}{V.~C. Le and K. Cools}

\title{Boundary element methods for the magnetic field integral equation on polyhedra\thanks{Submitted to the editors DATE.
\funding{This work was funded by the European Research Council (ERC) under the European Union's Horizon 2020 research and innovation programme (Grant agreement No. 101001847).}}}

\author{Van Chien Le\thanks{IDLab, Department of Information Technology, Ghent University - imec, 9000 Ghent, Belgium 
  (\email{vanchien.le@ugent.be}, \email{kristof.cools@ugent.be}).}
\and Kristof Cools\footnotemark[2]}

\ifpdf
\hypersetup{
  pdftitle={Boundary element methods for the magnetic field integral equation on polyhedra},
  pdfauthor={V.~C. Le and K. Cools}
}
\fi

\begin{document}

\maketitle

\begin{abstract}
This paper provides a rigorous analysis of boundary element methods for the magnetic field integral equation on Lipschitz polyhedra. The magnetic field integral equation is widely used in practical applications to model electromagnetic scattering by a perfectly conducting body. The governing operator is shown to be coercive by means of the electric field integral operator with a purely imaginary wave number. Consequently, the continuous variational problem is uniquely solvable, given that the wave number does not belong to the spectrum of the interior Maxwell's problem. A Petrov-Galerkin discretization scheme is then introduced, employing Raviart-Thomas boundary elements for the solution space and Buffa-Christiansen boundary elements for the test space. Under a mild assumption depending only on the geometrical domain, the corresponding discrete inf-sup condition is proven, implying the unique solvability of the discrete problem. An asymptotically quasi-optimal error estimate for numerical solutions is established, and the convergence rate of the numerical scheme is examined. In addition, the resulting matrix system is shown to be well-conditioned regardless of the mesh refinement. Some numerical results are presented to support the theoretical analysis.
\end{abstract}



\section{Introduction}

The numerical computation of electromagnetic fields scattered by a perfectly conducting body $\Om \sst \R^3$ usually leads to solving a Dirichlet boundary-value problem of the electric wave equation in the exterior unbounded domain $\Om^c := \R^3 \setminus \ovl{\Om}$, supplemented with the Silver-M\"{u}ller radiation condition at infinity. Based on the Stratton-Chu representation formula, it boils down to solving one of two standard boundary integral equations, namely the electric field integral equation (EFIE) and the magnetic field integral equation (MFIE).

With boundary element methods (BEMs) in mind, we assume that the domain $\Om$ is a polyhedron with Lipschitz continuous boundary $\Gm := \pa\Om$ that only consists of flat faces $\Gm_j$, $j = 1, 2, \ldots, N_F$. The extension of the results to Lipschitz curvilinear polyhedra is straightforward. On polyhedral boundary $\Gm$, the definition of Sobolev spaces $\Hs^s(\Gm)$, in particular with $\abs{s} \ge 1$, and differential operators are not obvious as those on smooth surfaces. By localization to each face $\Gm_j$, Buffa and Ciarlet successfully developed functional setting for Maxwell's equations on Lipschitz polyhedra, for which tangential traces of vector fields in $\HHs(\curlt, \Om)$, their trace spaces and relevant surface differential operators were appropriately characterized, see \cite{Buffa2001,BC2001,BC2001b}. All ingredients necessary to develop a standard functional theory, including integration by parts formulas and Hodge decompositions, were also provided. A generalization for general Lipschitz domains was then presented in \cite{BCS2002a}.

Based on the developed functional theory for Maxwell's equations on Lipschitz domains, a natural BEM for the EFIE on polyhedra was studied in \cite{HS2003b}. The continuous and discrete variational problems are coercive with respect to continuous and semi-discrete Hodge decompositions, respectively. The Galerkin discretization of the EFIE employing Raviart-Thomas boundary elements was analyzed, exhibiting the quasi-optimal convergence of numerical solutions. Another discretization scheme for the EFIE on non-smooth Lipschitz domains was proposed in \cite{BCS2002}, using standard boundary elements rather than Raviart-Thomas elements. On polyhedral surfaces, an asymptotically quasi-optimal convergence behaviour was also achieved. BEMs for Maxwell transmission problems on Lipschitz domains were investigated in \cite{BHV+2003}. The Galerkin discretization with Raviart-Thomas spaces was considered, and the results were applied to electromagnetic scattering by perfect electric conductors and by dielectric bodies. In the case of the perfect conductor scattering, only the EFIE was examined, also yielding quasi-optimal numerical solutions. To determine the actual rate of convergence on polyhedral domains, the authors of \cite{BCS2002} and \cite{HS2003b} all took into account the singularities of solution to Maxwell's equations, which were investigated in \cite{CD2000}. More specifically, the convergence rate of BEMs for the EFIE is limited by the regularity parameter $s^\ast$ of the Laplace-Beltrami operator on $\Gm$, which depends on the geometry of $\Gm$ in neighborhoods of vertices only.
                 
In addition to the EFIE, BEMs for combined field integral equations (CFIEs) on Lipschitz domains were investigated in \cite{BH2005,LC2024}. When the wave number belongs to the spectrum of the interior Maxwell's problem, the EFIE and MFIE are not uniquely solvable, resulting in resonant instabilities of Galerkin discretization systems. CFIEs owe their name to an appropriate combination of the Maxwell single and double layer potentials, thereby getting rid of the spurious resonances. In \cite{BH2005}, the authors exploited the coercivity of the EFIE operator by applying a compact regularization to the double layer part. The continuous and discrete variational problems were proven to be uniquely solvable at all wave numbers. A Galerkin boundary element discretization was proposed, whose numerical solutions satisfy an asymptotically quasi-optimal error estimate. The paper \cite{LC2024} introduced a preconditioned CFIE on which the EFIE and MFIE operators are preconditioned by their counterparts of a purely imaginary wave number. By means of a Calder\'{o}n projection formula, the governing CFIE operator was cast into an equivalent formulation that, up to a compact perturbation, is $S_{i\kappa}$-coercive, with $S_{i\kappa}$ the EFIE operator of purely imaginary wave number $i\kappa$. A Galerkin discretization of the mixed variational problem using the lowest-order Raviart-Thomas space and its dual Buffa-Christiansen space was analyzed, producing quasi-optimal numerical solutions. Moreover, the resulting matrix systems were shown to be well-conditioned on fine meshes and stable at resonant frequencies. Nonetheless, the proposed discretization scheme in \cite{LC2024} is not the one typically used by practitioners. In practice, the MFIE part of CFIEs is commonly discretized using Raviart-Thomas boundary elements for the solution space and Buffa-Christiansen elements for the test space. For some examples, the reader is referred to \cite{BCB+2013,MBC+2020,LCA+2023,LCA+2024}.

This paper aims to provide a rigorous analysis of BEMs for the MFIE on Lipschitz polyhedra. Like the EFIE and CFIEs, the MFIE is widely used in practical applications to model the electromagnetic scattering by perfect electric conductors. Despite its utility, however, attention paid to the MFIE is not at the same level with those for the EFIE and CFIEs. This is evident in the relatively limited number of papers devoting to BEMs for the MFIE, with most of them contributed by practitioners, see, e.g., \cite{RUP2001,DW2005,EG2005,CAD+2011,KE2018,KE2023}. On smooth surfaces, the MFIE operator is well-known to be a Fredholm operator of second kind \cite{BH2003}. The existence and uniqueness of a solution to the continuous variational problem of the MFIE has been shown for smooth domains, provided that the wave number is not an eigenvalue of the interior Maxwell's problem, see, e.g., \cite{BH2003}. In the discrete setting, the existence does not necessarily hold true, since the boundary element subspace may not be dual to itself. More precisely, the Galerkin discretization of the identity operator (the principal part of the MFIE operator) using lowest-order Raviart-Thomas elements has been proven to not satisfy the discrete inf-sup condition, see \cite[Section~3.1]{CN2002}. As a result, the coercivity of the corresponding Galerkin discretization of the MFIE does not hold true.

The contribution of this paper is twofold. First, we prove that the continuous variational problem of the MFIE is also uniquely solvable on Lipschitz domains, of course provided that the wave number does not belong to the spectrum of the interior Maxwell's problem. More particularly, we show that the MFIE operator is $S_{i\kappa^\prime}$-coercive up to a compact perturbation, for any $\kappa^\prime > 0$. It is noteworthy that on non-smooth Lipschitz surfaces, the average of exterior and interior traces of the double layer potential operator is not necessarily compact. Secondly, we analyze the mixed discretization scheme for the MFIE introduced in \cite{CAD+2011}, when the boundary $\Gm$ is Lipschitz continuous. The authors of \cite{CAD+2011} proposed a Petrov-Galerkin discretization of the MFIE employing Raviart-Thomas boundary elements for the solution space and Buffa-Christiansen boundary elements for the test space, on which they exploited the duality between the Buffa-Christiansen space and the Raviart-Thomas space \cite{BC2007}. Given that the average double layer boundary integral operator is compact on smooth surfaces, this duality supplies the coercivity. When the surface $\Gm$ is solely Lipschitz continuous, the coercivity of the mixed discretization remains valid if an additional stability condition is satisfied, which only depends on the boundary $\Gm$. This condition is mild as it is fulfilled for typical geometries in practical applications. As a consequence, the asymptotic quasi-optimality of numerical solutions is obtained. In addition, we show that the resulting matrix system is well-conditioned regardless of the boundary mesh refinement. To end this paragraph, we state here the essential assumption on the wave number of Maxwell's equations and stop repeating it in the remainder of the paper.
\begin{assumption}
\label{ast:wavenumber}
    The wave number is bounded away from the spectrum of the interior Maxwell's problem.
\end{assumption}

This paper is organized as follows. The next section recalls the standard functional setting for Maxwell's equations on Lipschitz polyhedra, including the tangential trace operators, trace spaces, relevant tangential differential operators and an integration by parts formula. Section~\ref{sec:potentials} gives the definition and properties of Maxwell potentials and boundary integral operators. The MFIE and its appropriate variational formulation are derived in Section~\ref{sec:mfie}. The governing MFIE operator is shown to be compact perturbation of a $S_{i\kappa^\prime}$-coercive operator, implying the unique solvability of the continuous variational problem. In Section~\ref{sec:discretization}, we analyze a Galerkin boundary element discretization of the MFIE. Its solvability under a stability assumption, an asymptotically quasi-optimal error estimate for discrete solutions, and the bounded condition number of the matrix system are established in this section. The numerical verification of the stability condition is discussed in Section~\ref{sec:verification}. Section~\ref{sec:results} presents some numerical results aiming to support the numerical analysis. Finally, we provide some conclusion remarks and discuss some possibilities for future works in Section~\ref{sec:conclusions}.

\section{Function spaces}
\label{sec:spaces}

In this section, we briefly recall the functional setting for Maxwell's equations on Lipschitz polyhedra. For more details and proofs, we refer the reader to the articles \cite{Buffa2001,BC2001,BC2001b,BCS2002a}.
Throughout the paper, regular symbols are used for scalar functions and their spaces, while bold symbols are used for vector fields and vector spaces.

Let $\Om \sst \R^3$ be a bounded polyhedron with Lipschitz continuous boundary $\Gm$, which only consists of flat faces $\Gm_j$, $j = 1, 2, \ldots, N_F$, i.e., $\ovl{\Gm} = \bigcup_{j} \ovl{\Gm_j}$. The unbounded complement $\Om^c := \R^3 \setminus \ovl{\Om}$ and the unit normal vector $\nv \in \LLs^{\infty}(\Gm)$ pointing from $\Om$ to $\Om^c$. On the polyhedral boundary $\Gm$, the usual Sobolev spaces $\Hs^s(\Gm)$ are defined invariantly for $s \in [0, 1]$, with convention $\Hs^0(\Gm) \equiv \Ls^2(\Gm)$ \cite[Section~1.3.3]{Grisvard1985}. For $s > 1$, the spaces $\Hs^s(\Gm)$ can be characterized face by face, with a careful treatment at their common edges, see \cite{BC2001,BC2001b}. We denote by $\Hs^{-s}(\Gm)$ the dual space of $\Hs^s(\Gm)$, and $\inprod{\cdot, \cdot}_{s, \Gm}$ their duality pairing. In addition, we introduce the space $\wt{\Hs}^{3/2}(\Gm) := \gm(\Hs^2(\Om))$ and its dual $\wt{\Hs}^{-3/2}(\Gm)$, with $\gm$ the standard trace operator\footnote{The tilde is used to distinguish the trace of $\Hs^2(\Om)$ on the boundary $\Gm$ and the space $\Hs^{3/2}(\Gm)$, as they are not identical in the case of a non-smooth Lipschitz domain $\Om$.}. The duality pairing between $\wt{\Hs}^{-3/2}(\Gm)$ and $\wt{\Hs}^{3/2}(\Gm)$ is denoted by $\inprod{\cdot, \cdot}_{\sim, 3/2, \Gm}$. Here, we make the usual identification $\Ls^2(\Gm) \sst \Hs^{-s}(\Gm)$ and $\Ls^2(\Gm) \sst \wt{\Hs}^{-3/2}(\Gm)$ via the $\Ls^2(\Gm)$-pairing. 

Next, we introduce the tangential trace operator $\gm_D$, which is defined for smooth vector fields $\ub \in \CCs^{\infty}(\ovl{\Om})$ by
\[
    (\gm_D \ub)(\xb) := \ub(\xb) \times \nv(\xb),
\]
for almost all $\xb \in \Gm$. This operator can be extended to a continuous mapping from $\HHs(\curlt, \Om)$ to $\HHs^{-1/2}(\Gm)$, see \cite[Theorem~1.5.1.1]{Grisvard1985}. In order to determine the actual range of $\gm_D$, we adopt the definition $\HHs^{s}_{\times}(\Gm) := \gm_D(\HHs^{s+ 1/2}(\Om))$, with $s \in (0, 1)$, from \cite[Definition~2.1]{BH2005}. For $s = 0$, we set the space $\HHs^0_{\times}(\Gm)$ identical with
\[
     \LLs^2_\tv(\Gm) :=  \brac{\ub \in \LLs^2(\Gm) : \ub \cdot \nv = 0}, 
\]
which is endowed with the anti-symmetric pairing
\[
    \inprod{\ub, \vb}_{\times, \Gm} := \int_\Gm (\ub \times \nv) \cdot \vb \ds, \qqqq \fa \ub, \vb \in \LLs^2_\tv(\Gm).
\]
The dual space of $\HHs^s_{\times}(\Gm)$ is denoted by $\HHs^{-s}_{\times}(\Gm)$, whose elements are identified via the $\LLs^2_\tv(\Gm)$-pairing $\inprod{\cdot, \cdot}_{\times, \Gm}$.

Let us introduce relevant surface differential operators. The surface operator $\curlt_\Gm : \wt{\Hs}^{3/2}(\Gm) \to \HHs^{1/2}_{\times}(\Gm)$ can be defined by \cite[Proposition~3.4]{BCS2002a}
\[
    \curlt_\Gm \gm(\varphi) = \gm_D(\gradt \varphi), \qqqqq \fa \varphi \in \Hs^2(\Om).
\]
Then, the surface divergence operator $\divt_\Gm : \HHs^{-1/2}_{\times}(\Gm) \to \wt{\Hs}^{-3/2}(\Gm)$ is defined as dual to $\curlt_\Gm$ in the sense
\[
    \inprod{\divt_\Gm \ub, \varphi}_{\sim, 3/2, \Gm} = - \inprod{\ub, \curlt_\Gm \varphi}_{1/2, \Gm}, \qqq \fa \ub \in \HHs^{-1/2}_{\times}(\Gm), \varphi \in \wt{\Hs}^{3/2}(\Gm).
\]
Now, we are able to introduce the space
\[
    \HHs^{-1/2}_{\times}(\divt_{\Gm}, \Gm) := \brac{\ub \in \HHs^{-1/2}_{\times}(\Gm) : \divt_{\Gm} \ub \in \Hs^{-1/2}(\Gm)},
\]
equipped with the graph norm
\[
    \norm{\ub}^2_{\HHs^{-1/2}_{\times}(\divt_{\Gm}, \Gm)} := \norm{\ub}^2_{\HHs^{-1/2}_{\times}(\Gm)} + \norm{\divt_\Gm \ub}^2_{\Hs^{-1/2}(\Gm)}.
\]

The next theorem presents the extension of the $\LLs^2_\tv(\Gm)$-pairing $\langle{\cdot, \cdot}\rangle_{\times, \Gm}$ to a bilinear form on the space $\HHs^{-1/2}_{\times}(\divt_{\Gm}, \Gm)$ and its self-duality with respect to this pairing, see, e.g., \cite{BC2001b} and \cite[Lemma~5.6]{BCS2002a}.

\begin{theorem}[self-duality of the space $\HHs^{-1/2}_{\times}(\divt_{\Gm}, \Gm)$]
\label{thm:duality}
    The pairing $\langle{\cdot, \cdot}\rangle_{\times, \Gm}$ can be extended to a continuous bilinear form on the space $\HHs^{-1/2}_{\times}(\divt_{\Gm}, \Gm)$. Moreover, $\HHs^{-1/2}_{\times}(\divt_{\Gm}, \Gm)$ becomes its own dual with respect to $\langle{\cdot, \cdot}\rangle_{\times, \Gm}$.
\end{theorem}

The tangential trace operator $\gm_D$ maps from $\HHs(\curlt, \Om)$ onto $\HHs^{-1/2}_{\times}(\divt_{\Gm}, \Gm)$, making $\HHs^{-1/2}_{\times}(\divt_{\Gm}, \Gm)$ the natural trace space for electromagnetic fields. The following theorem provides an integration by parts formula for functions in $\HHs(\curlt, \Om)$, which is taken from \cite[Theorem~4.1]{BCS2002a} and \cite[Theorem~2.3]{CH2012}.

\begin{theorem}[integration by parts formula]
    \label{thm:integration_by_parts}
    The tangential trace operator $\gm_D$ continuously maps from $\HHs(\curlt, \Om)$ onto $\HHs^{-1/2}_{\times}(\divt_{\Gm}, \Gm)$, which possesses a continuous right inverse. In addition, the following integration by parts formula holds
    \begin{equation}
    \label{eq:identity}
        \int_\Om (\curlt \, \ub \cdot \vb - \ub \cdot \curlt \, \vb) \dx = -\inprod{\gm_D \ub, \gm_D \vb}_{\times, \Gm}, \qq \fa \ub, \vb \in \HHs(\curlt, \Om).
    \end{equation}
\end{theorem}

When investigating the convergence rate of BEMs, higher regularities of the solution are required. The definition of the spaces $\HHs^s_{\times}(\divt_{\Gm}, \Gm)$, with $s > -\tfrac{1}{2}$, can be found in \cite{HS2003b,BH2005}. We omit those definitions for the sake of conciseness. Finally, we introduce the Neumann trace operator $\gm_N := \gm_D \circ \curlt$, which continuously maps from $\HHs(\curlt^2, \Om)$ onto $\HHs^{-1/2}_{\times}(\divt_{\Gm}, \Gm)$, see \cite{BH2005}.

\section{Boundary integral operators}
\label{sec:potentials}

In this paper, we consider the scattering of time-harmonic electromagnetic fields with angular frequency $\om > 0$. The exterior region $\Om^c$ is assumed to be filled by a homogeneous, isotropic and linear material with constant permittivity $\epsilon > 0$ and permeability $\mu > 0$. The wave number is defined by $\kappa := \om \sqrt{\epsilon\mu} > 0$. In aid of further analysis with integral operators of a purely imaginary wave number, we start from the ``the electric wave equation'' and the Silver-M\"{u}ller radiation condition with generic wave number $\sm$, which can be either $\kappa$ or $i\kappa^\prime$, with $\kappa, \kappa^\prime > 0$
\begin{align}
    & \,\, \curlt \, \curlt \, \eb - \sm^2 \eb = \zrb \qqqqqq \text{in} \q \Om \cup \Om^c, \label{eq:wave1} \\
    & \lim_{r \to \infty} \int_{\pa B_r} \abs{\curlt \, \eb \times \nv + i\sm (\nv \times \eb) \times \nv}^2 \ds = 0. \label{eq:SilverMuller1}
\end{align}
Here, $\eb$ is the unknown complex-valued vector field, $i$ stands for the imaginary unit, and $B_r$ denotes the ball of radius $r > 0$ centered at $\zrb$. 

When $\sm = \kappa$, the problem \eqref{eq:wave1}-\eqref{eq:SilverMuller1} in conjunction with the Dirichlet boundary condition
\begin{equation}
    \eb \times \nv = - \eb^{in} \times \nv \qqqq \text{on} \q \Gm, \label{eq:DBC}
\end{equation}
describes the scattering by a perfect electric conductor at the boundary $\Gm$ \cite[Section~6.4]{Colton1992}. In this case, the unknown $\eb$ represents the scattered electric field, and the incident electric field $\eb^{in}$ plays the role of excitation. The exterior problem has a unique solution $\eb$ for all wave number $\kappa$ according to Rellich's lemma \cite{Cessenat1996,Nedelec2001}, whereas the interior problem is only uniquely solvable if $\kappa^2$ does not coincide with one of the Dirichlet or Neumann eigenvalues of the $\curlt \, \curlt$-operator inside $\Om$.

When $\sm = i\kappa^\prime$, the equation \eqref{eq:wave1} is called the Yukawa-type equation. More details on the properties of boundary integral operators associated with this equation can be found in \cite[Section~5.6.4]{Nedelec2001}, \cite[Section~5.1]{Hiptmair2003b} and \cite{SW2009}. 

We denote by $G_\sm(\xb, \yb)$ the fundamental solution associated with $\Delta + \sm^2$, i.e.,
\[
    G_\sm(\xb, \yb) := \dfrac{\exp(i\sm\abs{\xb - \yb})}{4\pi \abs{\xb - \yb}}, \qqqq \xb \neq \yb.
\]
Then, the scalar and vectorial single layer potentials are respectively defined by
\[
    \Psi^\sm_V(\varphi)(\xb) := \int_\Gm \varphi(\yb) G_\sm(\xb, \yb) \ds(\yb), \qq \Psib^\sm_{\Ab}(\ub)(\xb) := \int_\Gm \ub(\yb) G_\sm(\xb, \yb) \ds(\yb),
\]
with $\xb \notin \Gm$. In electromagnetics, we are more interested in the Maxwell single and double layer potentials 
\[
    \Psib^\sm_{SL}(\ub) := \Psib^\sm_{\Ab} (\ub) + \dfrac{1}{\sm^2} \gradt \, \Psi^\sm_V(\divt_\Gm \ub), \qqq \Psib^\sm_{DL}(\ub) := \curlt\, \Psib^\sm_{\Ab}(\ub),
\]
which continuously map $\HHs^{-1/2}_{\times}(\divt_\Gm, \Gm)$ into $\HHs_{\textup{loc}}(\curlt^2, \Om \cup \Om^c)$, see \cite[Theorem~17]{Hiptmair2003b}. 
For any $\ub \in \HHs^{-1/2}_{\times}(\divt_\Gm, \Gm)$, the potentials $\Psib^{\sm}_{SL}(\ub)$ and $\Psib^{\sm}_{DL}(\ub)$ are solutions to the problem \eqref{eq:wave1}-\eqref{eq:SilverMuller1}. Therefore, it immediately holds that
\begin{equation}
\label{eq:sym1}
    \curlt \circ \Psib^\sm_{SL} = \Psib^\sm_{DL}, \qqqqq \curlt \circ \Psib^\sm_{DL} = \sm^2 \Psib^\sm_{SL}.
\end{equation}
On the other hand, every solution $\ub \in \HHs_{\textup{loc}}(\curlt^2, \Om \cup \Om^c)$ to the problem \eqref{eq:wave1}-\eqref{eq:SilverMuller1} satisfies the Stratton-Chu representation formula (see, e.g., \cite[Theorem~6.2]{Colton1992}, \cite[Theorem~5.5.1]{Nedelec2001} and \cite[Theorem~22]{Hiptmair2003b})
\begin{equation}
\label{eq:representation}
    \ub(\xb) = - \Psib^{\sm}_{DL}(\left[\gm_D\right]_{\Gm} \ub)(\xb) - \Psib^{\sm}_{SL}(\left[\gm_N\right]_{\Gm} \ub) (\xb), \qqq \xb \in \Om \cup \Om^c,
\end{equation}
where $\left[\gm\right]_{\Gm} := \gm^+ - \gm^-$ is the jump of some trace $\gm$ across the boundary $\Gm$, and the superscripts $-$ and $+$ designate traces onto $\Gm$ from $\Om$ and $\Om^c$, respectively.
In addition, the following jump relations hold true \cite[Theorem~18]{Hiptmair2003b}
\begin{equation}
    \label{eq:jump}
    \left[\gm_D\right]_\Gm \circ \Psib^\sm_{SL} = 0, \qqqqqq \left[\gm_N\right]_\Gm \circ \Psib^\sm_{SL} = -Id,
\end{equation}
where $Id$ stands for the identity operator. 

Next, let us denote by $\brac{\gm}_\Gm := \frac{1}{2}(\gm^+ + \gm^-)$ the average of exterior and interior traces $\gm$ on the boundary $\Gm$. The following boundary integral operators are well-defined and continuous \cite{BHV+2003,Hiptmair2003b}. 
\begin{definition}
\label{def:intops}
    For $\sm = \kappa$ or $\sm = i\kappa^\prime$, with $\kappa, \kappa^\prime > 0$, we define
    \begin{align*}
        & S_\sm := \brac{\gm_D}_\Gm \circ \Psib^\sm_{SL} && : \HHs^{-1/2}_{\times}(\divt_\Gm, \Gm) \to \HHs^{-1/2}_{\times}(\divt_\Gm, \Gm), \\
        & C_\sm := \brac{\gm_N}_\Gm \circ \Psib^\sm_{SL} && : \HHs^{-1/2}_{\times}(\divt_\Gm, \Gm) \to \HHs^{-1/2}_{\times}(\divt_\Gm, \Gm).
    \end{align*}
\end{definition}
By means of the relations \eqref{eq:sym1}, the tangential and Neumann traces of the Maxwell double layer potential $\Psib^\sm_{DL}$ can be derived accordingly. Then, the jump relations \eqref{eq:jump} imply that
\[
    \gm_D^{\pm} \circ \Psib^\sm_{SL} = \dfrac{1}{\sm^2} \gm_N^{\pm} \circ \Psib^\sm_{DL} = S_{\sm}, \qqq 
    \gm^{\pm}_N \circ \Psib^\sm_{SL} = \gm_D^{\pm} \circ \Psib^\sm_{DL} = \mp \dfrac{1}{2} Id + C_\sm.
\]

The following compactness is essential for further analysis of the MFIE, as it allows us to primarily work with the operator $C_{i\kappa^\prime}$ and then transfer to $C_\kappa$ by a compact perturbation argument. This result can be obtained by following the approach in the proof of \cite[Theorem~3.4.1]{Nedelec2001} or \cite[Theorem~3.12]{BHV+2003}, based on the regularity of the kernel $G_\kappa - G_{i\kappa^\prime}$.

\begin{lemma}
\label{lem:compact1}
    For any $\kappa, \kappa^\prime > 0$, the operator $C_\kappa - C_{i\kappa^\prime}$ is compact.
\end{lemma}

Now, we examine the properties of the boundary integral operators with purely imaginary wave number $\sm = i\kappa^\prime$, namely $S_{i\kappa^\prime}$ and $C_{i\kappa^\prime}$. The following lemma is taken from \cite[Theorem~21]{Hiptmair2003b} and \cite[Theorem~2.5]{SW2009}.

\begin{lemma}[$\HHs_\times^{-1/2}(\divt_\Gm, \Gm)$-ellipticity of $S_{i\kappa^\prime}$]
\label{lem:ellipticity}
    For $\kappa^\prime > 0$, the integral operator $S_{i\kappa^\prime}$ is symmetric with respect to $\langle{\cdot, \cdot}\rangle_{\times, \Gm}$ and $\HHs_\times^{-1/2}(\divt_\Gm, \Gm)$-elliptic, i.e., there exists a constant $C > 0$ depending only on $\kappa^\prime$ and $\Gm$ such that
    \begin{align*}
        & \inprod{\vb, S_{i\kappa^\prime} \ub}_{\times, \Gm} = \inprod{\ub, S_{i\kappa^\prime} \vb}_{\times, \Gm}, && \fa \ub, \vb \in \HHs^{-1/2}_\times(\divt_\Gm, \Gm), \\
        & \inprod{\ub, S_{i\kappa^\prime} \ovl{\ub}}_{\times, \Gm} \ge C \norm{\ub}^2_{\HHs^{-1/2}_\times(\divt_\Gm, \Gm)}, && \fa \ub \in \HHs^{-1/2}_\times(\divt_\Gm, \Gm).
    \end{align*}
\end{lemma}
The symmetry and ellipticity of $S_{i\kappa^\prime}$ imply that the space $\XXs_{\kappa^\prime} := \HHs^{-1/2}_\times(\divt_\Gm, \Gm)$ equipped with the inner product
\[
    \paren{\ub, \vb}_{\XXs_{\kappa^\prime}} := \inprod{\ub, S_{i\kappa^\prime} \ovl{\vb}}_{\times, \Gm}, \qqqqq \ub, \vb \in \HHs^{-1/2}_\times(\divt_\Gm, \Gm),
\]
is a Hilbert space. Its induced norm is denoted by $\norm{\cdot}_{\XXs_{\kappa^\prime}}$. Please note that the two spaces $\XXs_{\kappa^\prime}$ and $\HHs^{-1/2}_\times(\divt_\Gm, \Gm)$ only differ in their equipped norms. In what follows, the space $\XXs_{\kappa^\prime}$ is only used when its inner product and norm are involved or for a compact notation, whereas $\HHs^{-1/2}_\times(\divt_\Gm, \Gm)$ is used as the space of their elements (independent of $\kappa^\prime$).

The next results concern the contraction properties of the operators $2C_{i\kappa^\prime}$ and $\tfrac{1}{2} Id \pm C_{i\kappa^\prime}$. The latter were shown in \cite{SW2009} by extending the idea in \cite[Theorem~3.1]{SW2001} for Laplace's equation. In the following, we prove the contraction of $2C_{i\kappa^\prime}$ with respect to the norm $\norm{\cdot}_{\XXs_{\kappa^\prime}}$ by means of the Calder\'{o}n projection formulas for Maxwell's equations.
\begin{lemma}[Contraction property]
    \label{eq:contraction}
    For any $\kappa^\prime > 0$, there holds that
    \begin{align}
        & \inprod{\vb, C_{i\kappa^\prime} \ub}_{\times, \Gm} = \inprod{\ub, C_{i\kappa^\prime} \vb}_{\times, \Gm}, && \fa \ub, \vb \in \HHs^{-1/2}_\times(\divt_\Gm, \Gm), \label{eq:symmetry_Cik} \\
        & \norm{C_{i\kappa^\prime} \ub}_{\XXs_{\kappa^\prime}} \le \beta \norm{\ub}_{\XXs_{\kappa^\prime}}, && \fa \ub \in \HHs^{-1/2}_\times(\divt_\Gm, \Gm), \label{eq:contraction_Cik}
    \end{align}
    where the constant $\beta > 0$ is defined by
    \begin{equation}
    \label{eq:beta}
        \beta(\kappa^\prime, \Gm) := \sqrt{\dfrac{1}{4} - \dfrac{{\kappa^\prime}^2}{C_S^2}} < \dfrac{1}{2}, \qqq
        C_S(\kappa^\prime, \Gm) := \sup_{\ub \in \XXs_{\kappa^\prime}} \dfrac{\norm{S^{-1}_{i\kappa^\prime} \ub}_{\XXs_{\kappa^\prime}}}{\norm{\ub}_{\XXs_{\kappa^\prime}}}.
    \end{equation}
\end{lemma}
\begin{proof}
    The symmetry of $C_{i\kappa^\prime}$ with respect to $\inprod{\cdot, \cdot}_{\times, \Gm}$ can be obtained by following the lines in the proof of \cite[Theorem~3.9]{BHV+2003}, with the help of the integration by parts formula \eqref{eq:identity}. For the contraction property of $2C_{i\kappa^\prime}$, we need the following Calder\'{o}n projection formulas, which are derived from the Stratton-Chu representation formula \eqref{eq:representation} for $\sm = i\kappa^\prime$ and \cite[Formulation~35]{BH2003}
    \begin{align*}
        C_{i\kappa^\prime} \circ S_{i\kappa^\prime} & = -S_{i\kappa^\prime} \circ C_{i\kappa^\prime}, \\
        C_{i\kappa^\prime} \circ C_{i\kappa^\prime} & = \dfrac{1}{4} Id + {\kappa^\prime}^2 S_{i\kappa^\prime} \circ S_{i\kappa^\prime}.
    \end{align*}
    Based on these identities and the symmetry of $S_{i\kappa^\prime}$ and $C_{i\kappa^\prime}$, we can deduce that
    \begin{align*}
        \norm{C_{i\kappa^\prime} \ub}^2_{\XXs_{\kappa^\prime}} 
        & = \inprod{C_{i\kappa^\prime} \ub, \paren{S_{i\kappa^\prime} \circ C_{i\kappa^\prime}} \ovl{\ub}}_{\times, \Gm} \\
        & = \inprod{\ub, \paren{C_{i\kappa^\prime} \circ C_{i\kappa^\prime} \circ S_{i\kappa^\prime}} \ovl{\ub}}_{\times, \Gm} \\
        & = \dfrac{1}{4} \inprod{\ub, S_{i\kappa^\prime} \ovl{\ub}}_{\times, \Gm} + {\kappa^\prime}^2 \inprod{\ub, \paren{S_{i\kappa^\prime} \circ S_{i\kappa^\prime} \circ S_{i\kappa^\prime}} \ovl{\ub}}_{\times, \Gm} \\
        & = \dfrac{1}{4} \norm{\ub}^2_{\XXs_{\kappa^\prime}} - {\kappa^\prime}^2 \norm{S_{i\kappa^\prime} \ub}^2_{\XXs_{\kappa^\prime}} \\
        & \le \paren{\dfrac{1}{4} - \dfrac{{\kappa^\prime}^2}{C_S^2}} \norm{\ub}^2_{\XXs_{\kappa^\prime}}.
    \end{align*}
\end{proof}
The contraction properties of $\tfrac{1}{2} Id \pm C_{i\kappa^\prime}$ with respect to $\norm{\cdot}_{\XXs_{\kappa^\prime}}$ can be achieved as a consequence of \eqref{eq:contraction_Cik}.
\begin{corollary}
    For any $\kappa^\prime > 0$, the operators $\tfrac{1}{2} Id \pm C_{i\kappa^\prime}$ are contractions with respect to the norm $\norm{\cdot}_{\XXs_{\kappa^\prime}}$. In particular, for all $\ub \in \HHs^{-1/2}_\times(\divt_\Gm, \Gm)$, it holds that
    \begin{equation}
        \label{eq:contraction_MFIE}
        \paren{\dfrac{1}{2} - \beta} \norm{\ub}_{\XXs_{\kappa^\prime}} \le \norm{\paren{\dfrac{1}{2} Id \pm C_{i\kappa^\prime}} \ub}_{\XXs_{\kappa^\prime}} \le \paren{\dfrac{1}{2} + \beta} \norm{\ub}_{\XXs_{\kappa^\prime}}.
    \end{equation}
\end{corollary}

The following $S_{i\kappa^\prime}$-coercivity of $\tfrac{1}{2} Id \pm C_{i\kappa^\prime}$ is crucial in proving the coercivity of the MFIE operator in the next section.
\begin{lemma}[$S_{i\kappa^\prime}$-coercivity]
    \label{lem:Sik_coercivity}
    For any $\kappa^\prime > 0$, there holds that
    \begin{equation}
    \label{eq:Sik_coercivity}
        \paren{\paren{\dfrac{1}{2} Id \pm C_{i\kappa^\prime}} \ub, \ub}_{\XXs_{\kappa^\prime}} \ge \paren{\dfrac{1}{4} - \beta^2} \norm{\ub}^2_{\XXs_{\kappa^\prime}}, \qq \fa \ub \in \HHs^{-1/2}_\times(\divt_\Gm, \Gm).
    \end{equation}
\end{lemma}
\begin{proof}
    For all $\ub \in \HHs^{-1/2}_\times(\divt_\Gm, \Gm)$, it holds that
    \begin{align*}
        \paren{\paren{\dfrac{1}{2} Id \pm C_{i\kappa^\prime}} \ub, \ub}_{\XXs_{\kappa^\prime}} & = \dfrac{1}{2} \paren{\ub, \ub}_{\XXs_{\kappa^\prime}} \pm \paren{C_{i\kappa^\prime} \ub, \ub}_{\XXs_{\kappa^\prime}} \\
        & \ge \dfrac{1}{2} \norm{\ub}^2_{\XXs_{\kappa^\prime}} - \norm{C_{i\kappa^\prime} \ub}_{\XXs_{\kappa^\prime}} \norm{\ub}_{\XXs_{\kappa^\prime}} \\
        & \ge \dfrac{1}{2} \norm{\ub}^2_{\XXs_{\kappa^\prime}} - \dfrac{1}{4} \norm{\ub}^2_{\XXs_{\kappa^\prime}} - \norm{C_{i\kappa^\prime} \ub}^2_{\XXs_{\kappa^\prime}} \\
        & \ge \paren{\dfrac{1}{4} - \beta^2} \norm{\ub}^2_{\XXs_{\kappa^\prime}}.
    \end{align*}
\end{proof}
\begin{remark}
\label{rem:eqnorm}
    The contraction properties \eqref{eq:contraction_Cik} and \eqref{eq:contraction_MFIE} remain valid when replacing $\norm{\cdot}_{\XXs_{\kappa^\prime}}$ by the equivalent norm $\norm{S^{-1}_{i\kappa^\prime}\cdot}_{\XXs_{\kappa^\prime}}$. The space $\HHs^{-1/2}_\times(\divt_\Gm, \Gm)$ equipped with this norm also defines a Hilbert space. In addition, analogous to \eqref{eq:Sik_coercivity}, the operators $\tfrac{1}{2} Id \pm C_{i\kappa^\prime}$ are also $S_{i\kappa^\prime}^{-1}$-coercive, i.e., for all $\ub \in \HHs^{-1/2}_\times(\divt_\Gm, \Gm)$, it holds that
    \[
        \inprod{\paren{S^{-1}_{i\kappa^\prime} \circ \paren{\dfrac{1}{2} Id \pm C_{i\kappa^\prime}}} \ub, \ovl{\ub}}_{\times, \Gm} \ge \paren{\dfrac{1}{4} - \beta^2} \norm{S^{-1}_{i\kappa^\prime} \ub}^2_{\XXs_{\kappa^\prime}}.
    \]
\end{remark}

\section{The magnetic field integral equation}
\label{sec:mfie}

In this section, we derive the MFIE and show the unique solvability of its continuous variational formulation. Taking the exterior Neumann trace $\gm_N^+$ of the Stratton-Chu representation formula \eqref{eq:representation} with wave number $\sm = \kappa$ gives the standard MFIE of the form
\begin{equation}
    \label{eq:mfie}
    \paren{\dfrac{1}{2} Id + C_\kappa} (\ub) = \fb,
\end{equation}
where $\ub := \gm_N^+ \eb \in \HHs^{-1/2}_\times(\divt_\Gm, \Gm)$ is the unknown exterior Neumann trace of the scattered electric field, and $\fb \in \HHs^{-1/2}_\times(\divt_\Gm, \Gm)$ is given via the Dirichlet boundary condition \eqref{eq:DBC} as
\[
    \fb := \kappa^2 S_\kappa (\gm^+_D \eb^{in}).
\]

\subsection{Flawed idea}

One may think of deriving the variational formulation of \eqref{eq:mfie} using the $\XXs_{\kappa^\prime}$-inner product, i.e.,
\begin{equation}
    \label{eq:flawed_vf}
    \inprod{\paren{\dfrac{1}{2} Id + C_\kappa} \ub, S_{i\kappa^\prime} \ovl{\vb}}_{\times, \Gm} = \inprod{\fb, S_{i\kappa^\prime} \ovl{\vb}}_{\times, \Gm}, \qqq \fa \vb \in \HHs^{-1/2}_\times(\divt_\Gm, \Gm).
\end{equation}
Due to Lemmas~\ref{lem:compact1} and \ref{lem:Sik_coercivity}, the governing MFIE operator is compact perturbation of a $\XXs_{\kappa^\prime}$-elliptic operator, for any $\kappa^\prime > 0$. Given that it is injective (i.e., Assumption~\ref{ast:wavenumber} is satisfied), the variational formulation \eqref{eq:flawed_vf} has a unique solution $\ub \in \HHs^{-1/2}_\times(\divt_\Gm, \Gm)$. Moreover, the Galerkin method converges quasi-optimally when applied to this formulation, see \cite[Theorem~4.2.9]{Sauter2011}.

Despite that, the equation \eqref{eq:flawed_vf} involves three boundary integrals rather than two as in standard variational formulations for boundary integral equations (e.g., for EFIE in \cite{HS2003b}). Computing the entries of the Galerkin matrices for \eqref{eq:flawed_vf} requires the computation of integrals over all of $\Gm$, even when the basis functions have support on small subsets of $\Gm$. This disadvantage leads the assembling time of \eqref{eq:flawed_vf} even exceeding the time for solving it, making the formulation \eqref{eq:flawed_vf} undesired in practice.

\subsection{Variational formulation}

The appropriate variational formulation of \eqref{eq:mfie} is with respect to the duality pairing $\inprod{\cdot, \cdot}_{\times, \Gm}$: find $\ub \in \HHs^{-1/2}_\times(\divt_\Gm, \Gm)$ such that 
\begin{equation}
    \label{eq:vf}
    b(\ub, \vb) = \inprod{\fb, \ovl{\vb}}_{\times, \Gm}, \qqqqq \fa \vb \in \HHs^{-1/2}_\times(\divt_\Gm, \Gm),
\end{equation}
where sesquilinear form $b : \HHs^{-1/2}_\times(\divt_\Gm, \Gm) \times \HHs^{-1/2}_\times(\divt_\Gm, \Gm) \to \C$ is defined by
\[
    b(\ub, \vb) := \inprod{\paren{\dfrac{1}{2} Id + C_\kappa} \ub, \ovl{\vb}}_{\times, \Gm}, \qqqq \ub, \vb \in \HHs^{-1/2}_\times(\divt_\Gm, \Gm).
\]

On smooth surfaces, it is well-known that the operator $C_\kappa$ is compact, see, e.g., \cite[Section~5.4]{Nedelec2001} or \cite[Lemma~11]{BH2003}. As a result, the self-duality of the space $\HHs^{-1/2}_\times(\divt_\Gm, \Gm)$ with respect to $\inprod{\cdot, \cdot}_{\times, \Gm}$ ensures the coercivity of the sesquilinear form $b(\cdot, \cdot)$. More precisely, for any $\kappa^\prime > 0$, it holds that
\begin{equation}
    \label{eq:self_dual}
    \sup_{\vb \in \XXs_{\kappa^\prime}} \dfrac{\abs{\inprod{\ub, \ovl{\vb}}_{\times, \Gm}}}{\norm{S^{-1}_{i\kappa^\prime}\vb}_{\XXs_{\kappa^\prime}}} = \norm{\ub}_{\XXs_{\kappa^\prime}}, \qqqq \fa \ub \in \HHs^{-1/2}_\times(\divt_\Gm, \Gm).
\end{equation}
To prove this identity, one just need to take into account the fact that
\[
    \abs{\inprod{\ub, \ovl{\vb}}_{\times, \Gm}} = \abs{\inprod{\ub, \paren{S_{i\kappa^\prime} \circ S^{-1}_{i\kappa^\prime}} \ovl{\vb}}_{\times, \Gm}} \le \norm{\ub}_{\XXs_{\kappa^\prime}} \norm{S^{-1}_{i\kappa^\prime} \vb}_{\XXs_{\kappa^\prime}},
\]
for all $\ub, \vb \in \HHs^{-1/2}_\times(\divt_\Gm, \Gm)$, and that the equality occurs when $\vb = S_{i\kappa^\prime} \ub$. The relation \eqref{eq:self_dual} implies that both the inf-sup and the boundedness constants (with respect to the norms specified in \eqref{eq:self_dual}) of the duality pairing $\inprod{\cdot, \cdot}_{\times, \Gm}$ are 1.

On polyhedral boundaries, the compactness of $C_\kappa$ does not necessarily hold true. In this paper, we exploit the compactness of the operator $C_\kappa - C_{i\kappa^\prime}$ (cf. Lemma~\ref{lem:compact1}) and the $S_{i\kappa^\prime}$-coercivity stated in \eqref{eq:Sik_coercivity} to prove the coercivity of the sesquilinear form $b(\cdot, \cdot)$. For that reason, we split $b(\cdot, \cdot)$ into a principal part $a(\cdot, \cdot)$ and a compact sesquilinear form $t(\cdot, \cdot)$ as follows
\[
    a(\ub, \vb) := \inprod{\paren{\dfrac{1}{2} Id + C_{i\kappa^\prime}} \ub, \ovl{\vb}}_{\times, \Gm}, \qqq t(\ub, \vb) := \inprod{\paren{C_{\kappa} - C_{i\kappa^\prime}} \ub, \ovl{\vb}}_{\times, \Gm},
\]
for some $\kappa^\prime > 0$. The $S_{i\kappa^\prime}$-coercivity of the operator $\tfrac{1}{2} Id + C_{i\kappa^\prime}$ (cf. Lemma~\ref{lem:Sik_coercivity}) can be reformulated in terms of the sesquilinear form $a(\cdot, \cdot)$ as
\begin{equation}
\label{eq:coercivity_a}
    a(\ub, S_{i\kappa^\prime} \ub) \ge \paren{\dfrac{1}{4} - \beta^2} \norm{\ub}^2_{\XXs_{\kappa^\prime}}, \qqqq \fa \ub \in \HHs^{-1/2}_\times(\divt_\Gm, \Gm).
\end{equation}
The following generalized G{\aa}rding inequality for $b(\cdot, \cdot)$ is an immediate consequence of the inequality \eqref{eq:coercivity_a}.

\begin{theorem}[Generalized G{\aa}rding inequality]
    \label{lem:coercivity}
    For any $\kappa, \kappa^\prime > 0$, it holds that 
    \begin{equation}
    \label{eq:Garding}
        \abs{b(\ub, S_{i\kappa^\prime} \ub) - t(\ub, S_{i\kappa^\prime} \ub)} \ge  \paren{\dfrac{1}{4} - \beta^2} \norm{\ub}^2_{\XXs_{\kappa^\prime}}, \qq \fa \ub \in \HHs^{-1/2}_\times(\divt_\Gm, \Gm).
    \end{equation}
\end{theorem}

Finally, the unique solvability of the variational problem \eqref{eq:vf} follows from the injectivity of the MFIE (i.e., Assumption~\ref{ast:wavenumber}) and the generalized G{\aa}rding inequality \eqref{eq:Garding}, by means of a Fredholm alternative argument (cf. \cite[Proposition~3]{BCS2002}).
\begin{corollary}
    Let Assumption~\ref{ast:wavenumber} be satisfied. Then, there exists a unique $\ub \in \HHs^{-1/2}_{\times}(\divt_\Gm, \Gm)$ that solves \eqref{eq:vf}. 
\end{corollary}

\section{Galerkin discretization}
\label{sec:discretization}

In this section, we analyze the Galerkin boundary element discretization scheme for the variational formulation \eqref{eq:vf} that was firstly introduced in \cite{CAD+2011}. 

\subsection{Galerkin boundary element discretization}

Let $(\Gm_h)_{h > 0}$ be a family of shape-regular, quasi-uniform triangulations of the polyhedral surface $\Gm$ \cite{Ciarlet2002}. The parameter $h$ stands for the meshwidth, which is the length of the longest edge of triangulation $\Gm_h$. We denote by $\TT_h$ and $\EE_h$, respectively, the sets of all triangles and edges of $\Gm_h$. On each triangle $T \in \TT_h$, we equip the lowest-order triangular Raviart-Thomas space \cite{RT1977}
\[
    \RT_0(T) := \brac{\xb \mapsto \ab + b \xb : \ab \in \C^2, b \in \C}.
\]
This local space gives rise to the global $\divt_\Gm$-conforming boundary element space
\[
    \UUs_h := \brac{\ub_h \in \HHs^{-1/2}_{\times}(\divt_\Gm, \Gm) : \left.\ub_h\right|_T \in \RT_0(T), \, \fa T \in \TT_h},
\]
which is endowed with the edge degrees of freedom \cite{HS2003b}
\[
    \phi_e(\ub_h) := \int_e (\ub_h \times \nv_j) \cdot \ds,
\]
for all $e \in \EE_h$, where $\nv_j$ is the normal of a face $\Gm_j$ in whose the closure of $e$ is contained. The boundary element space $\UUs_h$ is also called the Rao-Wilton-Glisson space \cite{RWG1982}.

The Raviart-Thomas boundary elements have been widely used in the Galerkin discretization of the EFIE and CFIEs, both for the solution space and the test space, see, e.g., \cite{HS2003b,BH2005}. These discretization schemes are primarily based on the ellipticity of the operator $S_\kappa$. Unfortunately, this method fails when applied to the MFIE, even on smooth surfaces. It has been shown in \cite[Section~3.1]{CN2002} that there exist $C > 0$ and $\WWs_h \sst \UUs_h$ such that
\[
    \sup_{\ub_h \in \UUs_h} \sup_{\wb_h \in \WWs_h} \dfrac{\abs{\inprod{\ub_h, \ovl{\wb_h}}_{\times, \Gm}}}{\norm{\ub_h}_{\XXs_{\kappa^\prime}} \norm{S^{-1}_{i\kappa^\prime} \wb_h}_{\XXs_{\kappa^\prime}}} \le C h^{-1/2}.
\]
In other words, the subspace $\UUs_h$ is not dual to itself with respect to the pairing $\inprod{\cdot, \cdot}_{\times, \Gm}$.
In fact, there is ample numerical evidence that for a very broad class of situations that are of practical interest, the discretization matrix of the anti-symmetric pairing $\inprod{\cdot, \cdot}_{\times, \Gm}$ using the boundary element space $\UUs_h$ is singular, with a nullspace whose dimension is a sizeable fraction of the dimension of $\UUs_h$ \cite{CAO+2009b}.

In order to preserve the inf-sup condition for the Galerkin discretization of $\inprod{\cdot, \cdot}_{\times, \Gm}$, the authors of \cite{CAD+2011} proposed the test space $\VVs_h \sst \HHs^{-1/2}_{\times}(\divt_\Gm, \Gm)$ consisting of the Buffa-Christiansen boundary elements \cite{BC2007} (see Figure~\ref{fig:BC_function}). More particularly, $\dim \VVs_h = \dim \UUs_h$ and the following duality between $\UUs_h$ and $\VVs_h$ holds 
\begin{equation}
    \label{eq:discrete_dual}
    \inf_{\ub_h \in \UUs_h} \sup_{\vb_h \in \VVs_h} \dfrac{\abs{\inprod{\ub_h, \ovl{\vb_h}}_{\times, \Gm}}}{\norm{\ub_h}_{\XXs_{\kappa^\prime}} \norm{S^{-1}_{i\kappa^\prime}\vb_h}_{\XXs_{\kappa^\prime}}} \ge \alpha,
\end{equation}
for some constant $\alpha(\kappa^\prime, \Gm) > 0$ independent of $h$, with $\kappa^\prime > 0$. Obviously, $\alpha \le 1$ due to \eqref{eq:self_dual}. In what follows, $\alpha$ represents the maximal value that still fulfills \eqref{eq:discrete_dual}. Then, the Petrov-Galerkin boundary element discretization of the variational formulation \eqref{eq:vf} reads as: find $\ub_h \in \UUs_h$ such that
\begin{equation}
\label{eq:dis_vf}
    b(\ub_h, \vb_h) = \inprod{\fb, \ovl{\vb_h}}_{\times, \Gm}, \qqqqq \fa \vb_h \in \VVs_h.
\end{equation}

\begin{figure}
    \centering
    \begin{tikzpicture}
      \coordinate (A) at (0, 0);
      \coordinate (B) at (3, 0);
      \coordinate (C) at (1.7, 3.1);
      \coordinate (D) at (2.4, -2.9);
      \coordinate (E) at (-3, -0.3);
      \coordinate (F) at (-1.1, -2.8);
      \coordinate (G) at (-1.7, 2.5);
      \coordinate (H) at (5.2, 1.8);
      \coordinate (K) at (5.5, -1.5);

      \coordinate (AB) at (1.5, 0);
      \coordinate (BC) at (2.35, 1.55);
      \coordinate (AC) at (0.85, 1.55);
      \coordinate (AG) at (-0.85, 1.25);
      \coordinate (AD) at (1.2, -1.45);
      \coordinate (AE) at (-1.5, -0.15);
      \coordinate (AF) at (-0.55, -1.4);
      \coordinate (BD) at (2.7, -1.45);
      \coordinate (BH) at (4.1, 0.9);
      \coordinate (BK) at (4.25, -0.75);
      \coordinate (EF) at (-2.05, -1.55);
      \coordinate (EG) at (-2.35, 1.1);
      \coordinate (DF) at (0.65, -2.85);
      \coordinate (CG) at (0, 2.8);
      \coordinate (CH) at (3.45, 2.45);
      \coordinate (DK) at (3.95, -2.2);
      \coordinate (HK) at (5.35, 0.15);

      \coordinate (BC2) at (7.7/3, 3.1/3);
      \coordinate (AC2) at (1.7/3, 3.1/3);
      \coordinate (AG2) at (-1.7/3, 2.5/3);
      \coordinate (AD2) at (0.8, -2.9/3);
      \coordinate (AE2) at (-1, -0.1);
      \coordinate (AF2) at (-1.1/3, -2.8/3);
      \coordinate (BD2) at (2.8, -2.9/3);
      \coordinate (BH2) at (11.2/3, 0.6);
      \coordinate (BK2) at (11.5/3, -0.5);

      \coordinate (ABC) at (4.7/3, 3.1/3);
      \coordinate (ABD) at (1.8, -2.9/3);
      \coordinate (ACG) at (0, 5.6/3);
      \coordinate (AEG) at (-4.7/3, 2.2/3);
      \coordinate (AEF) at (-4.1/3, -3.1/3);
      \coordinate (ADF) at (1.3/3, -1.9);
      \coordinate (BCH) at (3.3, 4.9/3);
      \coordinate (BDK) at (10.9/3, -4.4/3);
      \coordinate (BHK) at (13.7/3, 0.1);

      \coordinate (AEG2) at (-9.4/9, 4.4/9);
      \coordinate (AEF2) at (-8.2/9, -6.2/9);
      \coordinate (ACG2) at (0, 11.2/9);
      \coordinate (ADF2) at (2.6/9, -3.8/3);
      \coordinate (ABC2) at (18.2/12, 3.1/12);
      \coordinate (ABC3) at (9.4/9, 6.2/9);
      \coordinate (ABC4) at (18.4/9, 6.2/9);
      \coordinate (ABD2) at (1.575, -2.9/12);
      \coordinate (ABD3) at (1.2, -5.8/9);
      \coordinate (ABD4) at (2.2, -5.8/9);
      \coordinate (BCH2) at (3.2, 9.8/9);
      \coordinate (BDK2) at (30.8/9, -8.8/9);
      \coordinate (BHK2) at (36.4/9, 0.2/3);

      \draw[line width=0.1mm, black!40, fill=black!20] (AE) -- (AEF) -- (AF) -- (ADF) -- (AD) -- (ABD) -- (BD) -- (BDK) -- (BK) -- (BHK) -- (BH) -- (BCH) -- (BC) -- (ABC) -- (AC) -- (ACG) -- (AG) -- (AEG) -- (AE);

     \draw[line width=0.1mm, black!40] (E) -- (AF) -- (D) -- (BK) -- (H) -- (BC) -- (A) -- (BD) -- (K) -- (BH) -- (C) -- (AG) -- (E);
      \draw[line width=0.1mm, black!40] (G) -- (AC) -- (B) -- (HK);
      \draw[line width=0.1mm, black!40] (G) -- (AE) -- (F) -- (AD) -- (B) -- (CH);
      \draw[line width=0.1mm, black!40] (EG) -- (A) -- (EF);
      \draw[line width=0.1mm, black!40] (CG) -- (A) -- (DF);
      \draw[line width=0.1mm, black!40] (C) -- (AB) -- (D);
      \draw[line width=0.1mm, black!40] (B) -- (DK);
      
      \draw[line width=0.5mm] (E) -- (F) -- (D) -- (K) -- (H) -- (C) -- (G) -- (E) -- (A) -- (B) -- (C) -- (A) -- (D) -- (B);
      \draw[line width=0.5mm] (G) -- (A) -- (F);
      \draw[line width=0.5mm] (H) -- (B) -- (K);
        
      \node[inner sep=2.5pt, draw, fill=white] at (A) {\footnotesize 1/12};
      \node[inner sep=2.5pt, draw, fill=white] at (B) {\footnotesize 1/10};
      \node[inner sep=2.5pt, draw, fill=white] at (AE2) {\footnotesize 0};
      \node[inner sep=2.5pt, draw, fill=white] at (AG2) {\footnotesize 2};
      \node[inner sep=2.5pt, draw, fill=white] at (AC2) {\footnotesize 4};
      \node[inner sep=2.5pt, draw, fill=white] at (AF2) {\footnotesize -2};
      \node[inner sep=2.5pt, draw, fill=white] at (AD2) {\footnotesize -4};
      \node[inner sep=2.5pt, draw, fill=white] at (BH2) {\footnotesize 1};
      \node[inner sep=2.5pt, draw, fill=white] at (BK2) {\footnotesize -1};
      \node[inner sep=2.5pt, draw, fill=white] at (BC2) {\footnotesize 3};
      \node[inner sep=2.5pt, draw, fill=white] at (BD2) {\footnotesize -3};
      \node[inner sep=2.5pt, draw, fill=white] at (AEG2) {\footnotesize 1};
      \node[inner sep=2.5pt, draw, fill=white] at (AEF2) {\footnotesize -1};
      \node[inner sep=2.5pt, draw, fill=white] at (ACG2) {\footnotesize 3};
      \node[inner sep=2.5pt, draw, fill=white] at (ADF2) {\footnotesize -3};
      \node[inner sep=2.5pt, draw, fill=white] at (BCH2) {\footnotesize 2};
      \node[inner sep=2.5pt, draw, fill=white] at (BHK2) {\footnotesize 0};
      \node[inner sep=2.5pt, draw, fill=white] at (BDK2) {\footnotesize -2};
      \node[inner sep=1.5pt, draw, fill=white] at (ABC2) {\scriptsize 1/2};
      \node[inner sep=2.5pt, draw, fill=white] at (ABC3) {\footnotesize 5};
      \node[inner sep=2.5pt, draw, fill=white] at (ABC4) {\footnotesize 4};
      \node[inner sep=1.5pt, draw, fill=white] at (ABD2) {\scriptsize -1/2};
      \node[inner sep=2.5pt, draw, fill=white] at (ABD3) {\footnotesize -5};
      \node[inner sep=2.5pt, draw, fill=white] at (ABD4) {\footnotesize -4};
    \end{tikzpicture}
    \caption{A Buffa-Christiansen basis element expressed as a linear combination of Raviart-Thomas basis elements on the barycentric refinement $\wt{\Gm}_h$ of the triangulation $\Gm_h$, with coefficient of each edge multiplied by the number indicated at its origin. The edges of the refinement $\wt{\Gm}_h$ are oriented away from the central edge.}
    \label{fig:BC_function}
\end{figure}
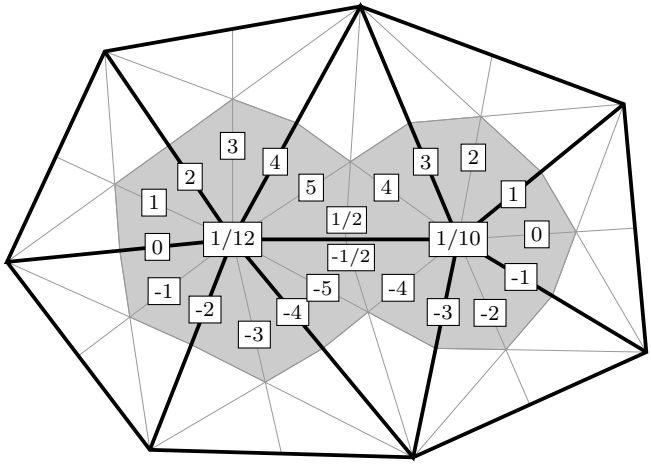

The following lemma presents the discrete inf-sup condition for the sesquilinear form $a(\cdot, \cdot)$ restricted on $\UUs_h \times \VVs_h$. This condition is a weaker discrete counterpart of the $S_{i\kappa^\prime}$-coercivity \eqref{eq:coercivity_a}.

\begin{lemma}
\label{lem:coercivity_a_h}
    For any $\kappa^\prime > 0$, the following inf-sup condition is satisfied
    \begin{equation}
    \label{eq:semi_elliptic_h}
        \inf_{\ub_h \in \UUs_h} \sup_{\vb_h \in \VVs_h} \dfrac{\abs{a(\ub_h, \vb_h)}}{\norm{\ub_h}_{\XXs_{\kappa^\prime}} \norm{S^{-1}_{i\kappa^\prime} \vb_h}_{\XXs_{\kappa^\prime}}} \ge \dfrac{\alpha}{2} - \beta.
    \end{equation}
\end{lemma}
\begin{proof}
    The following estimate holds true due to the discrete inf-sup condition \eqref{eq:discrete_dual} and the contraction property \eqref{eq:contraction_Cik}
    \begin{align*}
        \inf_{\ub_h \in \UUs_h} \sup_{\vb_h \in \VVs_h} \dfrac{\abs{a(\ub_h, \vb_h)}}{\norm{\ub_h}_{\XXs_{\kappa^\prime}} \norm{S^{-1}_{i\kappa^\prime} \vb_h}_{\XXs_{\kappa^\prime}}} & \ge \dfrac{1}{2} \inf_{\ub_h \in \UUs_h} \sup_{\vb_h \in \VVs_h} \dfrac{\abs{\inprod{\ub_h, \ovl\vb_h}_{\times, \Gm}}}{\norm{\ub_h}_{\XXs_{\kappa^\prime}} \norm{S^{-1}_{i\kappa^\prime} \vb_h}_{\XXs_{\kappa^\prime}}} \\
        & \q - \sup_{\ub_h \in \UUs_h} \sup_{\vb_h \in \VVs_h} \dfrac{\abs{\inprod{C_{i\kappa^\prime} \ub_h, \ovl\vb_h}_{\times, \Gm}}}{\norm{\ub_h}_{\XXs_{\kappa^\prime}} \norm{S^{-1}_{i\kappa^\prime} \vb_h}_{\XXs_{\kappa^\prime}}} \\
        & \ge \dfrac{\alpha}{2} - \sup_{\ub_h \in \UUs_h} \dfrac{\norm{C_{i\kappa^\prime} \ub_h}_{\XXs_{\kappa^\prime}}}{\norm{\ub_h}_{\XXs_{\kappa^\prime}}} \ge \dfrac{\alpha}{2} - \beta.
    \end{align*}
\end{proof}

The constant $\tfrac{\alpha}{2} - \beta$ is only dependent of $\kappa^\prime$ and $\Gm$, but independent of $h$. Moreover, it can be non-positive. In order to ensure that the inf-sup condition \eqref{eq:semi_elliptic_h} is satisfied with a positive constant, we make the following assumption.
\begin{assumption}[Stability condition]
\label{ast:inf_sup_const}
    There is a real number $\kappa^\prime_0 > 0$ such that
    \begin{equation}
        \label{eq:condition}
        \alpha(\kappa^\prime_0, \Gm) > 2 \beta(\kappa^\prime_0, \Gm),
    \end{equation}
    where $\alpha$ and $\beta$ are defined by \eqref{eq:discrete_dual} and \eqref{eq:beta}, respectively.
\end{assumption}

Now, we shall show the discrete inf-sup condition for the sesquilinear form $b : \UUs_h \times \VVs_h \to \C$, which is crucial in establishing the unique solvability of the discrete variational problem \eqref{eq:dis_vf}. The next result can be proven by following the hints in \cite[Exercise~26.5]{Ern2021b}, with the approximability of $\UUs_h$ and $\VVs_h$ in $\HHs^{-1/2}_{\times}(\divt_\Gm, \Gm)$ stated in \cite[Corollary~5]{BH2003} and \cite{BC2007}, respectively. We leave the complete proof for the reader.
\begin{lemma}
\label{lem:dis_inf_sup}
    Let $\kappa$ and $\kappa^\prime_0$ fulfill Assumptions~\ref{ast:wavenumber} and \ref{ast:inf_sup_const}, respectively. Then, there exist an $h_0 > 0$ and a constant $\gamma > 0$ such that for all $h < h_0$
    \begin{equation}
    \label{eq:dis_inf_sup}
        \inf_{\ub_h \in \UUs_h} \sup_{\vb_h \in \VVs_h} \dfrac{\abs{b(\ub_h, \vb_h)}}{\norm{\ub_h}_{\XXs_{\kappa_0^\prime}} \norm{S^{-1}_{i\kappa_0^\prime} \vb_h}_{\XXs_{\kappa_0^\prime}}} \ge \gamma.
    \end{equation}
\end{lemma}

The unique solvability of the discrete variational problem \eqref{eq:dis_vf} and the asymptotic quasi-optimality of the discrete solutions immediately follow from Lemma~\ref{lem:dis_inf_sup}, see \cite[Theorem~4.2.1]{Sauter2011}. 
\begin{theorem}
\label{thm:dis_solvability}
     Let Assumptions~\ref{ast:wavenumber} and \ref{ast:inf_sup_const} be satisfied. Then, there exists an $h_0 > 0$ such that for all $h < h_0$, the discrete problem \eqref{eq:dis_vf} has a unique solution $\ub_h \in \UUs_h$. In addition, the discrete solutions $\ub_h$ converge to the solution $\ub \in \HHs^{-1/2}_{\times}(\divt_\Gm, \Gm)$ of the problem \eqref{eq:vf} and satisfy the asymptotically quasi-optimal error estimate
    \begin{equation}
    \label{eq:quasioptimal_error}
        \norm{\ub - \ub_h}_{\HHs^{-1/2}_{\times}(\divt_\Gm, \Gm)} \le C \inf_{\wb_h \in \UUs_h} \norm{\ub - \wb_h}_{\HHs^{-1/2}_{\times}(\divt_\Gm, \Gm)},
    \end{equation}
    for some constant $C > 0$ depending only on $\Gm$ and $\kappa$. 
\end{theorem}

It is worth mentioning that the quasi-optimal error estimate \eqref{eq:quasioptimal_error} alone does not imply the actual convergence rate of the Galerkin discretization scheme \eqref{eq:dis_vf}. When investigating the convergence rate of a Galerkin discretization scheme for Maxwell's equations on polyhedral domains, we must take into account the singularities of the solutions, as shown by Costabel and Dauge in \cite{CD2000}. In particular, the error of best approximations by lowest-order Raviart-Thomas boundary elements was established in \cite[Lemma~8.1]{HS2003b}, which is limited by a regularity exponent $s^\ast > 0$ of the Laplace-Beltrami operator $\Delta_\Gm$. The parameter $s^\ast$ depends only on the geometry of $\Gm$ in neighborhoods of vertices. Moreover, there exist polyhedral vertices for which $s^\ast > 0$ is arbitrarily small, as demonstrated in \cite{BCS2002}. We refer the reader to \cite[Theorem~8]{BCS2002} and \cite[Theorem~2.1]{HS2003b} for more details on the regularity parameter $s^\ast$. Now, we are in the position to restate Lemma~8.1 in \cite{HS2003b}, which we exploit for our convergence analysis.

\begin{lemma}
    \label{lem:error}
    If $\zb \in \HHs^s_\times(\divt_\Gm, \Gm), s > 0$, then for any $\veps > 0$, it holds that
    \[
        \inf_{\wb_h \in \UUs_h} \norm{\zb - \wb_h}_{\HHs^{-1/2}_{\times}(\divt_\Gm, \Gm)} \le C h^{\min\brac{\tfrac{3}{2}-\veps, s+\tfrac{1}{2}-\veps, 1+s^\ast, s +s^\ast}} \norm{\zb}_{\HHs^s_\times(\divt_\Gm, \Gm)},
    \]
    with constant $C > 0$ depending only on $\Gm, s^\ast$ and $\veps$. 
\end{lemma}

Finally, we are able to establish the convergence rate of the Galerkin discretization \eqref{eq:dis_vf} of the MFIE, employing Raviart-Thomas boundary elements for the solution space and Buffa-Christiansen boundary elements for the test space.
\begin{theorem}
    \label{thm:convergence_rate}
    Let Assumptions~\ref{ast:wavenumber} and \ref{ast:inf_sup_const} be satisfied. In addition, we assume that the solution to \eqref{eq:vf} satisfies $\ub \in \HHs^s_\times(\divt_\Gm, \Gm), s > 0$. Then, there exists an $h_0 > 0$ such that for all $h < h_0$ and for any $\veps > 0$, the unique solution $\ub_h \in \UUs_h$ to \eqref{eq:dis_vf} satisfies
    \begin{equation}
    \label{eq:convergence_rate}
        \norm{\ub - \ub_h}_{\HHs^{-1/2}_{\times}(\divt_\Gm, \Gm)} \le C h^{\min\brac{\tfrac{3}{2}-\veps, s+\tfrac{1}{2}-\veps, 1+s^\ast, s +s^\ast}} \norm{\ub}_{\HHs^s_\times(\divt_\Gm, \Gm)},
    \end{equation}
    with constant $C > 0$ depending only on $\Gm, \kappa, s^\ast$ and $\veps$.
\end{theorem}

\subsection{Matrix representation}
Let $\{\ub_1, \ub_2, \ldots, \ub_N\}$ and $\{\vb_1, \vb_2, \ldots, \vb_N\}$ be bases of the boundary element spaces $\UUs_h$ and $\VVs_h$, respectively, with $N := \dim \UUs_h = \dim \VVs_h$. The solution $\ub_h \in \UUs_h$ to \eqref{eq:dis_vf} can be represented by its expansion coefficient vector $\wh{\ub}_h := (\wh{u}_1, \wh{u}_2, \ldots, \wh{u}_N)^{\transpose} \in \C^N$ such that
\begin{equation}
\label{eq:discrete_solution}
    \ub_h = \sum\limits_{i=1}^N \wh{u}_i \ub_i.
\end{equation}
Next, the Galerkin matrices of the sesquilinear form $b(\cdot, \cdot)$ and the pairing $\inprod{\cdot, \cdot}_{\times, \Gm}$ are respectively defined by
\[
   {[\BBs]}_{mn} := b(\ub_n, \vb_m), \qqqqq {[\DDs]}_{mn} := \inprod{\ub_n, \ovl{\vb_m}}_{\times, \Gm}, 
\]
with $m, n = 1, 2, \ldots, N$. Then, the coefficient vector $\wh{\ub}_h$ is the solution to the following matrix system 
\begin{equation}
\label{eq:matrix_system}
    \BBs \wh{\ub}_h = \bb,
\end{equation}
where the right-hand side vector ${[\bb]}_m := \inprod{\fb, \ovl{\vb_m}}_{\times, \Gm}$. The solvability of \eqref{eq:matrix_system} is a direct consequence of Theorem~\ref{thm:dis_solvability}.
\begin{corollary}
    Let Assumptions~\ref{ast:wavenumber} and \ref{ast:inf_sup_const} be satisfied. Then, there exists an $h_0 > 0$ such that for all $h < h_0$, the matrix equation \eqref{eq:matrix_system} has a unique solution $\wh{\ub}_h \in \C^N$. In addition, the corresponding $\ub_h \in \UUs_h$ defined by \eqref{eq:discrete_solution} is a quasi-optimal approximation to the solution $\ub$ of \eqref{eq:vf} in the sense of \eqref{eq:quasioptimal_error}.
\end{corollary}

Now, we investigate the conditioning property of the system \eqref{eq:matrix_system}. Let $\lambda_{\max}(\AAs)$ and $\lambda_{\min}(\AAs)$ be the maximal and minimal (by moduli) eigenvalues of a square matrix $\AAs$, respectively. The spectral condition number of $\AAs$ is then defined by
\[
    \cond(\AAs) := \dfrac{\abs{\lambda_{\max}(\AAs)}}{\abs{\lambda_{\min}(\AAs)}}.
\]
The following result shows that the problem \eqref{eq:matrix_system} is well-conditioned regardless of the boundary mesh refinement.
\begin{theorem}
    \label{thm:condition_number}
    Let Assumptions~\ref{ast:wavenumber} and \ref{ast:inf_sup_const} be satisfied. Then, there exist an $h_0 > 0$ and a constant $C > 0$ such that for all $h < h_0$
    \[
        \cond(\DDs^{-1} \BBs) \le C.
    \]  
\end{theorem}
\begin{proof}
    We follow the proof of \cite[Theorem~2.1]{Hiptmair2006}. Let the subspaces $\UUs_h$ and $\VVs_h$ be equipped with the norm $\norm{\cdot}_{\XXs_{\kappa^\prime_0}}$, where $\kappa^\prime_0$ fulfills Assumption~\ref{ast:inf_sup_const}. We denote by $B_h : \UUs_h \to \VVs_h^\prime$ and $D_h : \UUs_h \to \VVs_h^\prime$ the bounded linear operators associated with the sesquilinear form $b(\cdot, \cdot)$ and the pairing $\inprod{\cdot, \cdot}_{\times, \Gm}$, respectively. The following estimates hold true due to \eqref{eq:self_dual}, \eqref{eq:discrete_dual} and \eqref{eq:dis_inf_sup}
    \begin{align*}
        & \norm{D_h}_{\UUs_h \to \VVs_h^\prime} \le 1, &&  \norm{D_h^{-1}}_{\VVs_h^\prime \to \UUs_h} \le \alpha^{-1}, \\
        & \norm{B_h}_{\UUs_h \to \VVs_h^\prime} \le \norm{b}, && \norm{B_h^{-1}}_{\VVs_h^\prime \to \UUs_h} \le \gm^{-1}.
    \end{align*}
    Next, let $\C^N$ be equipped with the norm $\norm{\cdot}_{\XXs_{\kappa^\prime_0}}$ via the coefficient isomorphism with respect to the basis $\{\ub_1, \ub_2, \ldots, \ub_N\}$. Then, there hold that
    \begin{align*}
        \abs{\lambda_{\max}(\DDs^{-1} \BBs)} & \le \norm{\DDs^{-1} \BBs} \le \alpha^{-1} \norm{b}, \\
        \abs{\lambda_{\max}(\BBs^{-1} \DDs)} & \le \norm{\BBs^{-1} \DDs} \le \gamma^{-1}.
    \end{align*}
    Hence, we can conclude
    \[
        \cond(\DDs^{-1} \BBs) = \abs{\lambda_{\max}(\DDs^{-1} \BBs)} \abs{\lambda_{\max}(\BBs^{-1} \DDs)} \le \dfrac{\norm{b}}{\gamma \alpha}.
    \]
\end{proof}

\begin{remark}
    In the paper \cite{CS2022}, the authors introduced a family of 3D open-book Lipschitz polyhedra on which the Galerkin discretization of second-kind boundary integral equations for the Laplace and Helmholtz problems do not converge. The negative claims in that paper do not contradict the present results, since in this paper we consider the relevant operators in the natural trace space $\HHs^{-1/2}_\times(\divt_\Gm, \Gm)$ rather than the space $\LLs^2(\Gm)$ as in \cite{CS2022}.
\end{remark}

\begin{remark}
    The introduction of the purely imaginary wave number $i\kappa^\prime$ (and $\kappa^\prime_0$) is just to check the ``quality'' of $\Gm$ whether the proposed Galerkin method for the MFIE converges on that domain. Once Assumption~\ref{ast:inf_sup_const} is verified, the solvability of the discrete problem \eqref{eq:dis_vf}, the quasi-optimal error estimate and the convergence of numerical solutions are independent of $\kappa^\prime_0$. However, the bound of the condition number in Theorem~\ref{thm:condition_number} might be dependent of $\kappa^\prime_0$ via the constants $\alpha$ and $\gm$.
\end{remark}

\section{Verification of Assumption~\ref{ast:inf_sup_const}}
\label{sec:verification}

Assumption~\ref{ast:inf_sup_const}, along with Assumption~\ref{ast:wavenumber}, gives a sufficient condition for the unique solvability of the discrete problem \eqref{eq:dis_vf}. However, verifying Assumption~\ref{ast:inf_sup_const} for each domain is not straightforward. In this section, we propose a numerical scheme to approximately compute the constants $\alpha$ and $\beta$, then verify this assumption in practice.

We first evaluate the two terms in the denominator of \eqref{eq:discrete_dual}. For $\kappa^\prime > 0$, let $\SSs_{\kappa^\prime}$ be the Galerkin matrix of the $\XXs_{\kappa^\prime}$-inner product $\paren{\cdot, \cdot}_{\XXs_{\kappa^\prime}} : \UUs_h \times \UUs_h \to \C$. Then, the $\XXs_{\kappa^\prime}$-norm of $\ub_h \in \UUs_h$ can be expressed as the $\textbf{\textit{l}}^2$-norm of its coefficient vector $\wh{\ub}_h$ as follows
\begin{equation}
    \label{eq:l2_norm}
    \norm{\ub_h}_{\XXs_{\kappa^\prime}} = \sqrt{\wh{\ub}_h^\transpose \SSs_{\kappa^\prime} \wh{\ub}_h} = \norm{\SSs_{\kappa^\prime}^{1/2} \wh{\ub}_h}_{\textbf{\textit{l}}^2}.
\end{equation}

In order to evaluate $\norm{S^{-1}_{i\kappa^\prime} \vb_h}_{\XXs_{\kappa^\prime}}$, with $\vb_h \in \VVs_h$, we approximate $S^{-1}_{i\kappa^\prime} \vb_h$ by its projection $\wt{\wb}_h \in \wt{\UUs}_h$ with respect to the $\XXs_{\kappa^\prime}$-inner product, i.e., $\wt{\wb}_h \in \wt{\UUs}_h$ is the unique solution to the following equation
\begin{equation}
\label{eq:barycentric_projection}
    \inprod{\wt{\wb}_h, S_{i\kappa^\prime}  \ovl{\wt{\vphib}_h}}_{\times, \Gm} = - \inprod{\vb_h, \ovl{\wt{\vphib}_h}}_{\times, \Gm}, \qqqqq \fa \wt{\vphib}_h \in \wt{\UUs}_h.
\end{equation}
Here, the subspace $\wt{\UUs}_h$ consists of all Raviart-Thomas boundary elements defined on the barycentric refinement $\wt{\Gm}_h$ of $\Gm_h$. Please note that $\UUs_h \subset \wt{\UUs}_h, \VVs_h \subset \wt{\UUs}_h$, and $\dim \wt{\UUs}_h = 6\dim \UUs_h = 6\dim \VVs_h = 6N$. Let $\wt{\SSs}_{\kappa^\prime}$ and $\GGs$ be the Galerkin matrices of the sesquilinear forms defined by the left-hand side and right-hand side of \eqref{eq:barycentric_projection}, respectively. Then, the coefficient vector $\wh{\wb}_h$ of $\wt{\wb}_h$ with respect to the basis of $\wt{\UUs}_h$ is determined by 
\[
    \wh{\wb}_h = \wt{\SSs}_{\kappa^\prime}^{-1} \GGs \wh{\vb}_h,
\]
where $\wh{\vb}_h$ is the coefficient vector of $\vb_h \in \VVs_h$ with respect to the basis $\{\vb_1, \vb_2, \ldots, \vb_N\}$. Thus, $\norm{S^{-1}_{i\kappa^\prime} \vb_h}_{\XXs_{\kappa^\prime}}$ is approximated by 
\begin{equation}
    \label{eq:inverse_approx}
    \norm{S^{-1}_{i\kappa^\prime} \vb_h}_{\XXs_{\kappa^\prime}} \approx \norm{\wt{\wb}_h}_{\XXs_{\kappa^\prime}} = \sqrt{\wh{\wb}_h^\transpose \wt{\SSs}_{\kappa^\prime} \wh{\wb}_h} = \norm{\paren{\GGs^\transpose \wt{\SSs}_{\kappa^\prime}^{-1} \GGs}^{1/2} \wh{\vb}_h}_{\textbf{\textit{l}}^2}.
\end{equation}
Let $\MMs_{\kappa^\prime} := \GGs^\transpose \wt{\SSs}_{\kappa^\prime}^{-1} \GGs$. The matrix $\MMs_{\kappa^\prime}$ is positive-definite, as $\wt{\SSs}_{\kappa^\prime}$ is positive-definite and $\GGs$ is injective. Based on the matrix representations \eqref{eq:l2_norm} and \eqref{eq:inverse_approx}, the left-hand side of the inf-sup condition \eqref{eq:discrete_dual} can be expressed in the following form
\begin{align*}
    \inf_{\ub_h \in \UUs_h} \sup_{\vb_h \in \VVs_h} \dfrac{\abs{\inprod{\ub_h, \ovl{\vb_h}}_{\times, \Gm}}}{\norm{\ub_h}_{\XXs_{\kappa^\prime}} \norm{S^{-1}_{i\kappa^\prime}\vb_h}_{\XXs_{\kappa^\prime}}} & \approx \inf_{\wh{\ub}_h \in \C^N} \sup_{\wh{\vb}_h \in \C^N} \dfrac{\abs{\wh{\vb}_h^\transpose \DDs \wh{\ub}_h}}{\norm{\SSs_{\kappa^\prime}^{1/2} \wh{\ub}_h}_{\textbf{\textit{l}}^2} \norm{\MMs_{\kappa^\prime}^{1/2} \wh{\vb}_h}_{\textbf{\textit{l}}^2}} \\
    & = \inf_{\wh{\xib}_h \in \C^N} \sup_{\wh{\psib}_h \in \C^N} \dfrac{\abs{\wh{\psib}_h^\transpose \MMs_{\kappa^\prime}^{-\transpose/2} \DDs \SSs_{\kappa^\prime}^{-1/2} \wh{\xib}_h}}{\norm{\wh{\xib}_h}_{\textbf{\textit{l}}^2} \norm{\wh{\psib}_h}_{\textbf{\textit{l}}^2}},
\end{align*}
where $\MMs_{\kappa^\prime}^{-\transpose/2} = \paren{\MMs_{\kappa^\prime}^{-1/2}}^{\transpose}$. Hence, for a sufficiently small $h > 0$, the constant $\alpha$ in \eqref{eq:discrete_dual} can be approximated by
\[
    \alpha \approx \abs{\lambda_{\min} \paren{ \MMs_{\kappa^\prime}^{-\transpose/2} \DDs \SSs_{\kappa^\prime}^{-1/2}}}.
\]
Please bear in mind that the space $\C^N$ is now equipped with the Euclidean $\textbf{\textit{l}}^2$-norm. Analogously, taking into account the property \eqref{eq:contraction_Cik}, $\beta$ can be estimated as
\begin{equation}
\label{eq:beta_approx}
    \beta \approx \abs{\lambda_{\max} \paren{\MMs_{\kappa^\prime}^{-\transpose/2} \KKs_{\kappa^\prime} \SSs_{\kappa^\prime}^{-1/2}}},
\end{equation}
where $\KKs_{\kappa^\prime}$ is the Galerkin matrix associated with the operator $C_{i\kappa^\prime}$, i.e.,
\[
    {[\KKs_{\kappa^\prime}]}_{mn} := \inprod{\ub_n, C_{i\kappa^\prime} \ovl{\vb_m}}_{\times, \Gm}, \qqq m, n = 1, 2, \ldots, N.
\]
Finally, we can verify Assumption~\ref{ast:inf_sup_const} for the boundary $\Gm$ by checking if the condition $\tfrac{\alpha}{2} - \beta > 0$ is fulfilled by some $\kappa^\prime > 0$.

\begin{remark}
    Another scheme to approximate the constant $\beta$ is to replace the Galerkin matrix $\KKs_{\kappa^\prime}$ by its counterpart $\wt{\KKs}_{\kappa^\prime}$ tested with $\UUs_h$, i.e.,
    \[
        {[\wt{\KKs}_{\kappa^\prime}]}_{mn} := \inprod{\ub_n, C_{i\kappa^\prime} \ovl{\ub_m}}_{\times, \Gm}, \qqqq m, n = 1, 2, \ldots, N.
    \]
    Then, the matrices $\GGs$ and $\MMs_{\kappa^\prime}$ must also be changed accordingly (the latter is denoted by $\wt{\MMs}_{\kappa^\prime}$). The advantage of this approach is the symmetry of the matrix $\wt{\KKs}_{\kappa^\prime}$, due to \eqref{eq:symmetry_Cik}. As a consequence, all eigenvalues of $\wt{\MMs}_{\kappa^\prime}^{-\transpose/2} \wt{\KKs}_{\kappa^\prime} \SSs_{\kappa^\prime}^{-1/2}$ are real, and its maximal eigenvalue $\abs{\lambda_{\max}}$ might give a better approximation of $\beta$. The disadvantage is that we have to additionally compute the matrix $\wt{\MMs}_{\kappa^\prime}$ and its square root inverse, which is quite expensive. From authors' experience, these two approaches produce similar results. Therefore, the scheme \eqref{eq:beta_approx} is recommended.
\end{remark}

\begin{remark}
    From authors' experience, Assumption~\ref{ast:inf_sup_const} is satisfied for typical geometries in practical applications, usually with larger $\kappa^\prime > 0$ for less regular domains (see Figure~\ref{fig:verification} for an example). When discretizing the integral operators with larger $\kappa^\prime$, finer meshes and/or more quadrature points are required to accurately compute the non-zero entries of the corresponding Galerkin matrices. As a result, the verification step may cost more than the computational cost of solving the MFIE itself. Therefore, verifying Assumption~\ref{ast:inf_sup_const} beforehand is usually not desirable in practice. 
\end{remark}

\section{Numerical results}
\label{sec:results}

This section presents some numerical results to support the analysis of the Galerkin discretization \eqref{eq:dis_vf} of the MFIE. We perform numerical experiments for the electromagnetic scattering by two perfectly conducting bodies: a multiply-connected cuboid and a star-based pyramid, see Figure~\ref{fig:domains}. The verification of Assumption~\ref{ast:inf_sup_const} for these domains are firstly performed. The approximated values of $\alpha$ and $\beta$ with respect to different $\kappa^\prime$ are depicted in Figure~\ref{fig:verification}, which clearly shows that there are some values of $\kappa^\prime$ fulfilling Assumption~\ref{ast:inf_sup_const} for each domain (the values at which the blue curve is above the red one).

\begin{figure}
    \centering
    \begin{subfigure}[b]{0.4\textwidth}
        \centering
        \includegraphics[trim={33cm 15cm 33cm 28cm}, clip, width=\textwidth]{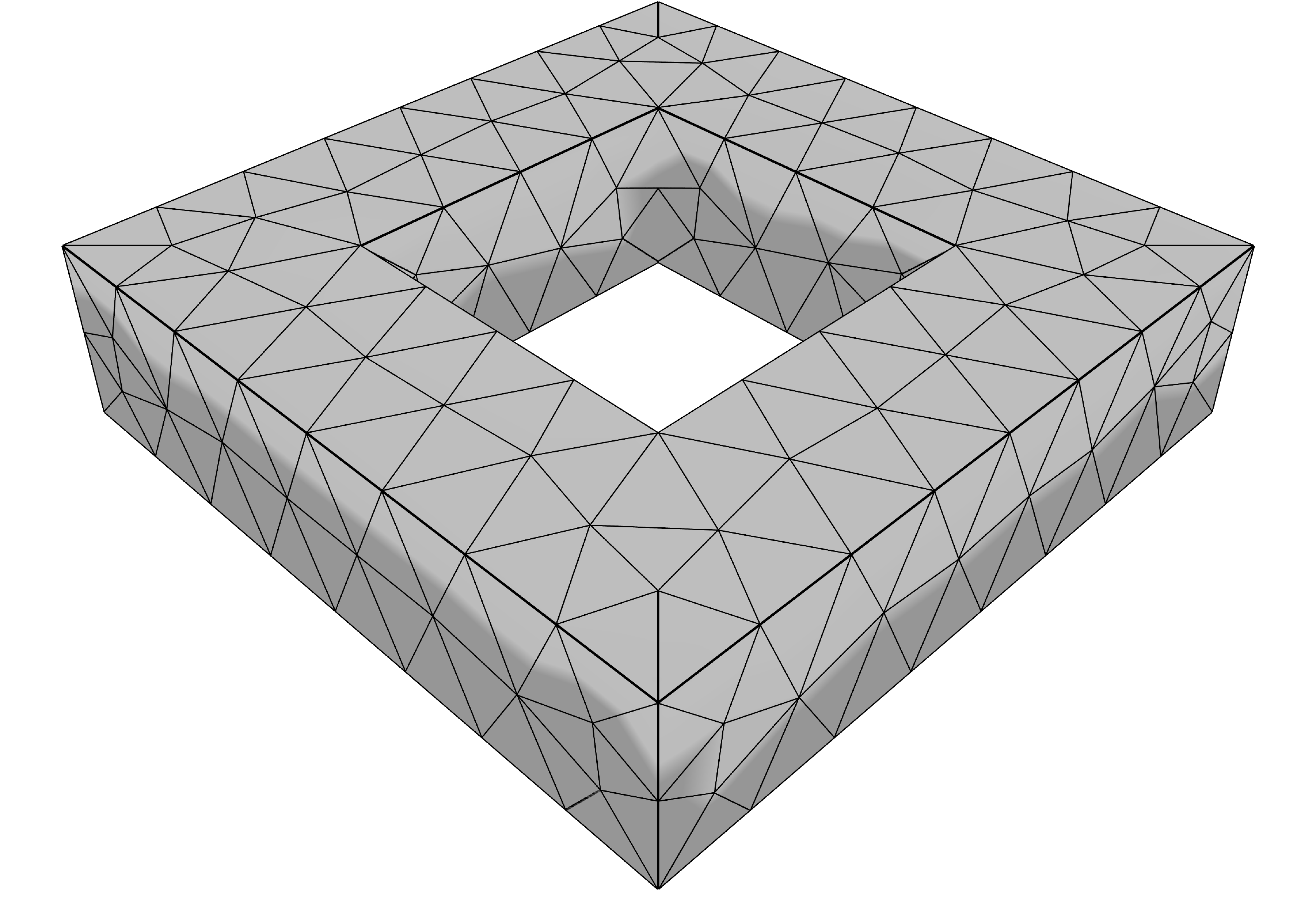}
    \end{subfigure}
    \hfill
    \begin{subfigure}[b]{0.52\textwidth}
        \centering
        \includegraphics[trim={34cm 29cm 34cm 17cm}, clip, width=0.8\textwidth]{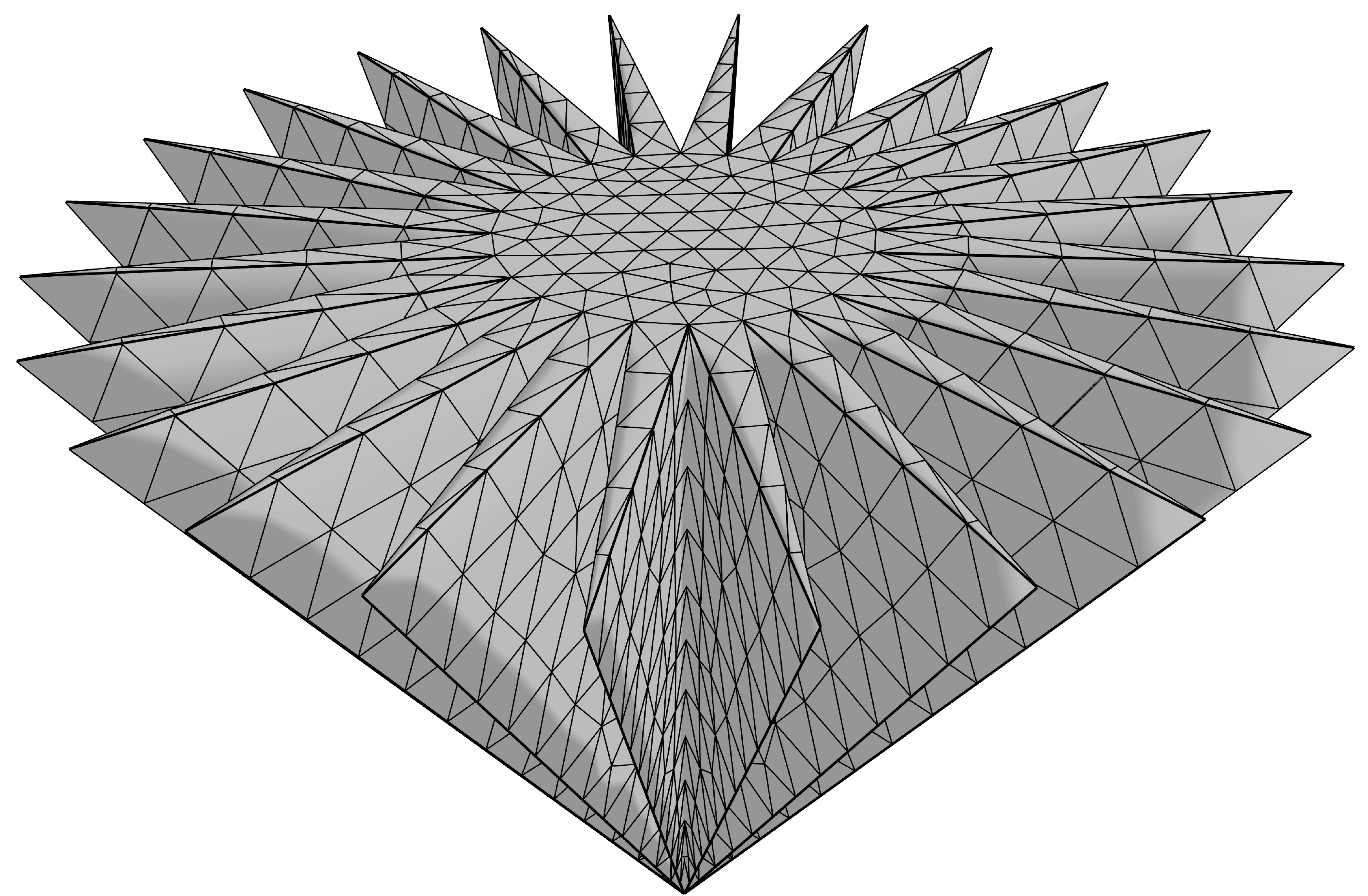}
        \hspace{-1cm}
        \includegraphics[trim={38cm 18cm 38cm 18cm}, clip, width=0.2\textwidth]{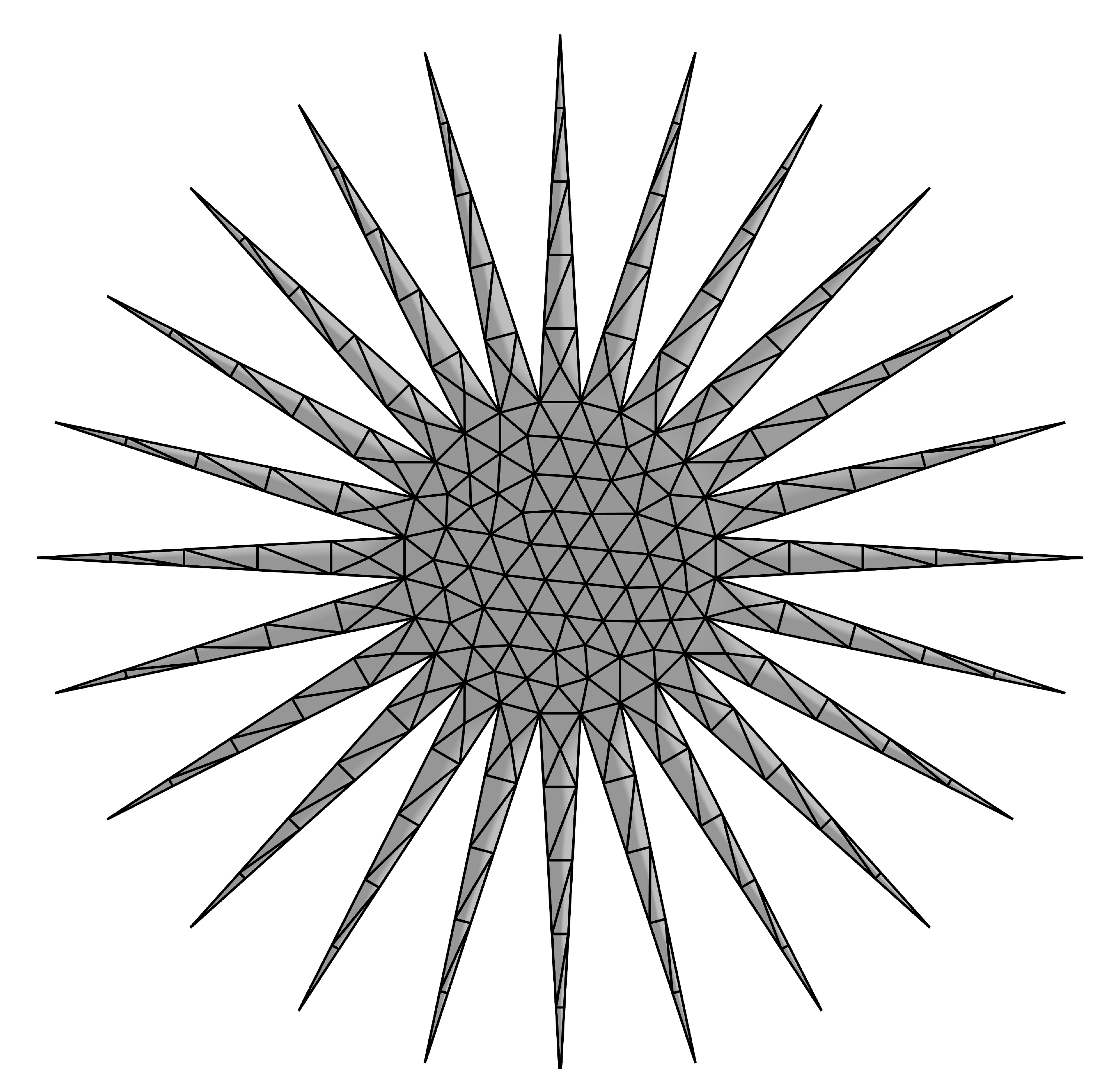}
    \end{subfigure}
    \caption{Triangulation of polyhedral boundaries used in numerical experiments. \textit{Left:} A cuboid of dimension $1 \mathrm{m} \times 1 \mathrm{m} \times 0.25 \mathrm{m}$ with a concentric cuboid hole of dimension $0.5 \mathrm{m} \times 0.5 \mathrm{m} \times 0.25 \mathrm{m}$. \textit{Right:} A pyramid of height $0.5 \mathrm{m}$ and 24-pointed star base (bottom-right corner), whose vertices lie on two concentric circles of radius $1 \mathrm{m}$ and $0.3 \mathrm{m}$.}
    \label{fig:domains}
\end{figure} 

\begin{figure}
    \centering
    \begin{subfigure}[b]{0.49\textwidth}
        \centering
        \includegraphics[width=\textwidth]{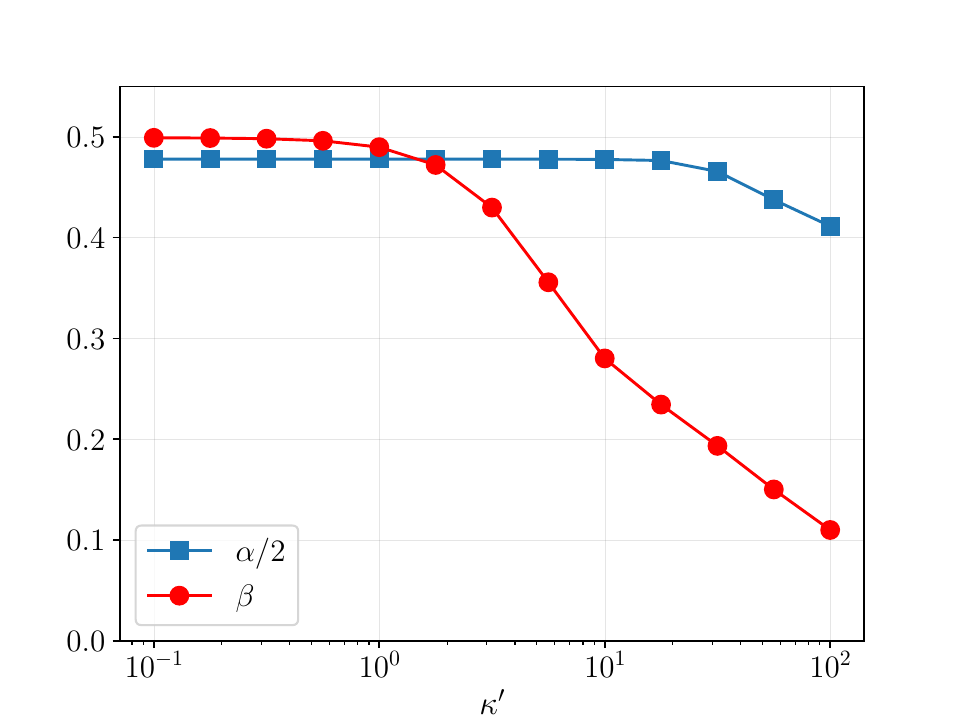}
    \end{subfigure}
    \hfill
    \begin{subfigure}[b]{0.49\textwidth}
        \centering \includegraphics[width=\textwidth]{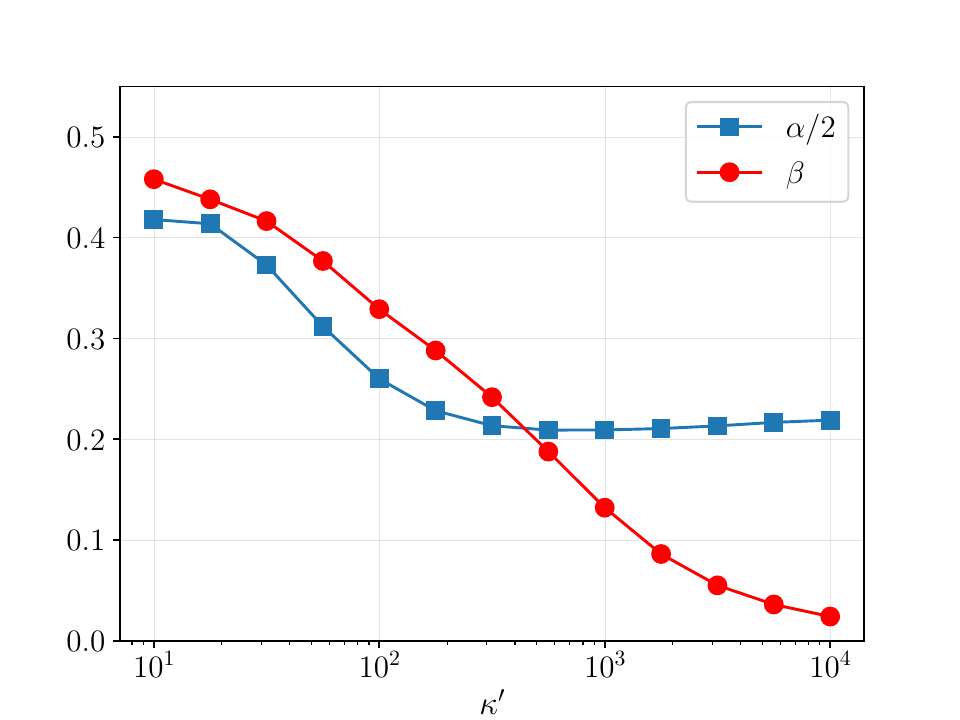}
    \end{subfigure}
    \caption{Approximated constants $\alpha$ and $\beta$ for the verification of Assumption~\ref{ast:inf_sup_const}. \textit{Left:} A cuboid of dimension $\mathrm{1\mathrm{m} \times 1\mathrm{m} \times 0.25\mathrm{m}}$ with a concentric cuboid hole of dimension $\mathrm{0.5\mathrm{m} \times 0.5\mathrm{m} \times 0.25\mathrm{m}}$. \textit{Right:} A pyramid of height $0.5 \mathrm{m}$ and a 24-pointed star base, whose vertices lie on two circles of radius $1\mathrm{m}$ and $0.3\mathrm{m}$. For each domain, there are some $\kappa^\prime$ at which Assumption~\ref{ast:inf_sup_const} is satisfied.}
    \label{fig:verification}
\end{figure}

In all further experiments, the wave number $\kappa = 1$ and the incident electric field is given by
\[
    \eb^{in}(\xb) = \hat{\xb} \exp(-i\kappa \, \hat{\zb} \cdot \xb),
\]
where $\hat{\xb}$ and $\hat{\zb}$ stand for the unit vectors along the $x$-axis and $z$-axis, respectively. In order to examine the convergence of the Galerkin discretization scheme \eqref{eq:dis_vf}, we define the relative error
\[
    \mathrm{err}_h := \dfrac{\abs{\norm{\ub_h}_{\XXs_{\kappa}} - \norm{\ub_{\mathrm{ref}}}_{\XXs_{\kappa}}}}{\norm{\ub_{\mathrm{ref}}}_{\XXs_{\kappa}}},
\]
where $\ub_{\mathrm{ref}}$ is a reference solution to \eqref{eq:dis_vf} computed with a very fine mesh. By means of a triangle inequality, the relative error $\mathrm{err}_h$ exhibits the same convergence rate with that stated in Theorem~\ref{thm:convergence_rate}. 

Next, we examine the well-conditioning of the matrix system \eqref{eq:matrix_system} by computing the condition number of the matrix $\DDs^{-1} \BBs$ for different meshwidth $h$. In addition, to investigate the performance of the proposed discretization with iterative solvers, the corresponding number of GMRES iterations required to solve $\eqref{eq:matrix_system}$ (with tolerance $\veps_0 = 10^{-8}$) is also reported. Even though the bounded condition number of $\DDs^{-1} \BBs$ does not imply a fast convergence of GMRES solvers, it shows a good prediction in the numerical experiments below as well as in practical applications. The condition number and the corresponding GMRES iteration counts of the Galerkin discretization of the MFIE are then compared with those of the EFIE
\begin{equation}
    \label{eq:EFIE}
    S_{\kappa} \ub = \paren{\dfrac{1}{2} Id + C_{\kappa}} \paren{\gm_D^+ \eb^{in}},
\end{equation}
which is derived by taking the exterior tangential trace $\gm_D^+$ of the Stratton-Chu representation formula \eqref{eq:representation} with $\sm = \kappa$. The EFIE \eqref{eq:EFIE} is discretized using lowest-order Raviart-Thomas boundary elements for both the solution space and the test space, as investigated in \cite{HS2003b}.

\subsection{Multiply-connected cuboid}
We start with the scattering by a multiply-connected cuboid of dimension $1\mathrm{m} \times 1 \mathrm{m} \times 0.25 \mathrm{m}$ with a concentric cuboid hole of dimension $0.5\mathrm{m} \times 0.5\mathrm{m} \times 0.25 \mathrm{m}$, see Figure~\ref{fig:domains} (\textit{left}). 

The reference solution $\ub_{\mathrm{ref}}$ is computed with the boundary mesh $\Gm_h$ characterized by the meshwidth $h = 0.0125 \mathrm{m}$. At each vertex of the boundary $\Gm$, there are only three edges meeting, with opening angles $\vphi_1, \vphi_2$ and $\vphi_3$ equal $\pi/2$ or $3\pi/2$. Thus, the regularity parameter $s^\ast = 2\pi/(\vphi_1 + \vphi_2 + \vphi_3) - \veps > 1/2$, see \cite{HS2003b}. It means that the convergence rate, in this case, only depends on the regularity exponent $s$ of the solution, but not on $s^{\ast}$. More specifically, one can obtain the optimal convergence rate $\OO(h^{1.5})$ for the Galerkin discretization \eqref{eq:dis_vf} of the MFIE on this domain, provided that the continuous solution is sufficiently regular, i.e, $\ub \in \HHs^s_\times(\divt_\Gm, \Gm), s \ge 1$. Figure~\ref{fig:error} (\textit{left}) confirms this assertion.
\begin{figure}
    \centering
    \begin{subfigure}[b]{0.49\textwidth}
        \centering
        \includegraphics[width=\textwidth]{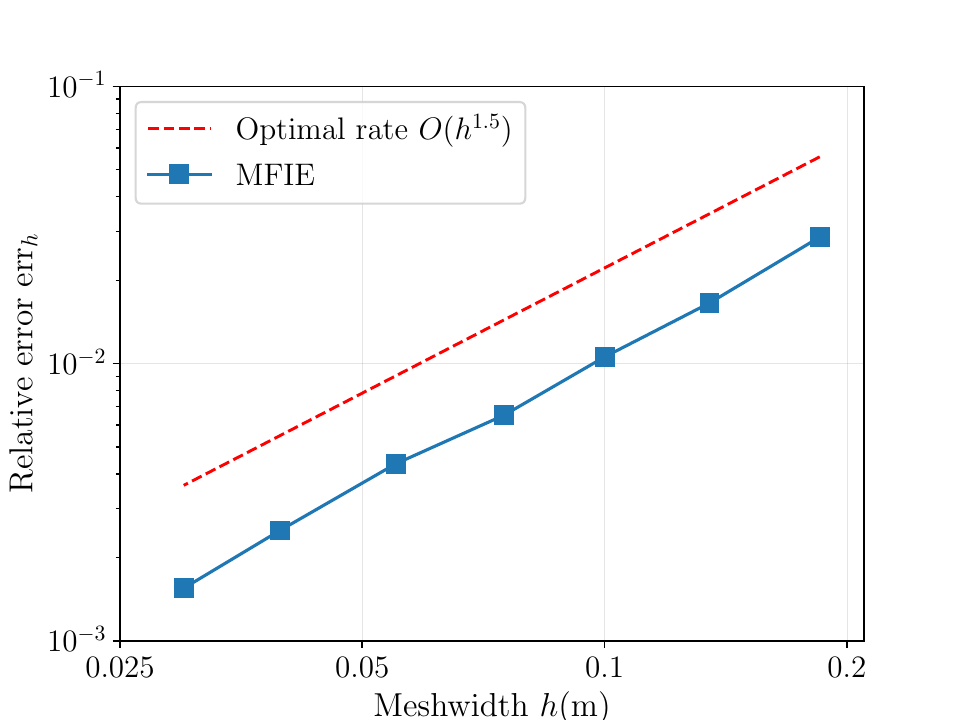}
    \end{subfigure}
    \hfill
    \begin{subfigure}[b]{0.49\textwidth}
        \centering \includegraphics[width=\textwidth]{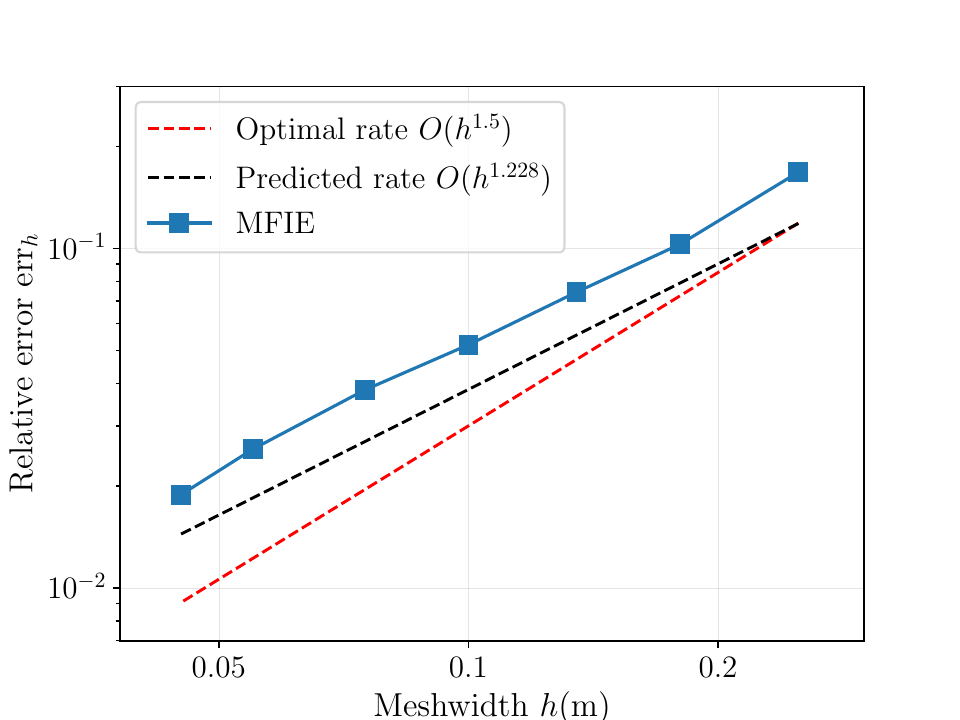}
    \end{subfigure}
    \caption{Relative error of numerical solutions to the MFIE with respect to meshwidth $h$. \textit{Left:} A cuboid of dimension $\mathrm{1\mathrm{m} \times 1\mathrm{m} \times 0.25\mathrm{m}}$ with a concentric cuboid hole of dimension $\mathrm{0.5\mathrm{m} \times 0.5\mathrm{m} \times 0.25\mathrm{m}}$. The numerical solutions converge to the reference solution with an optimal rate. \textit{Right:} A pyramid of height $0.5 \mathrm{m}$ and a 24-pointed star base, whose vertices lie on two circles of radius $1\mathrm{m}$ and $0.3\mathrm{m}$. The numerical solutions exhibit a sub-optimal convergence rate due to the singularities of solution at the bottom vertex.}
    \label{fig:error}
\end{figure}

Figure~\ref{fig:cond} (\textit{left}) depicts the conditioning properties of the Galerkin discretization of the EFIE and MFIE on the multiply-connected cuboid. Whereas the condition number of the Galerkin matrix of the EFIE and the corresponding number of GMRES iterations grow as $\OO(h^{-2})$, those of the MFIE stay almost constant. This result confirms the well-conditioning of the Galerkin discretization \eqref{eq:dis_vf} of the MFIE.

\begin{figure}
    \centering
    \begin{subfigure}[b]{0.49\textwidth}
        \centering
        \includegraphics[width=\textwidth]{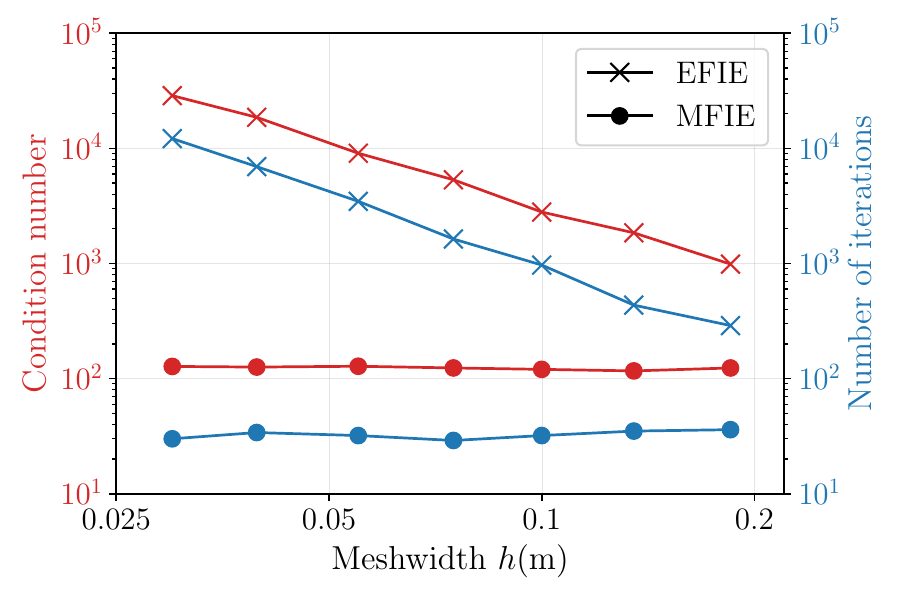}
    \end{subfigure}
    \hfill
    \begin{subfigure}[b]{0.49\textwidth}
        \centering \includegraphics[width=\textwidth]{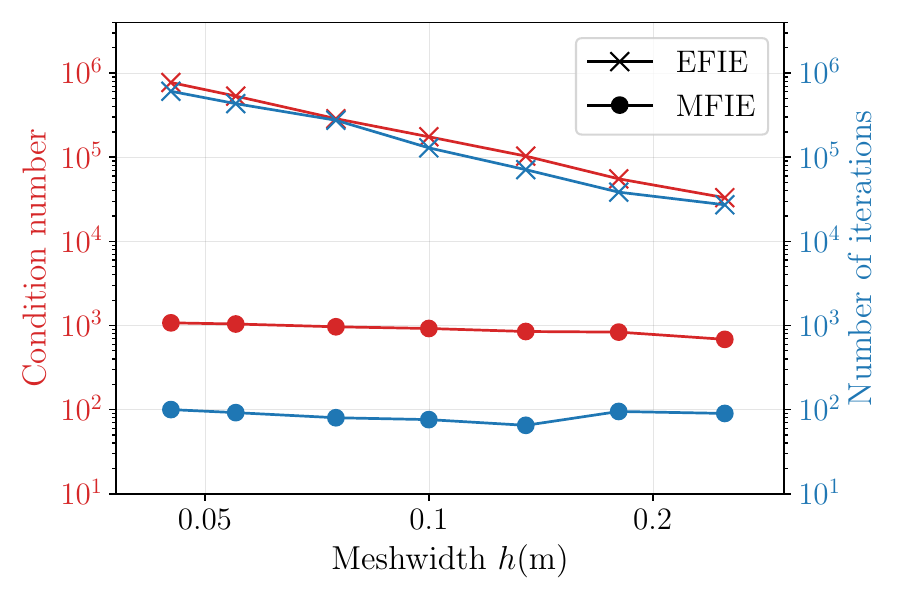}
    \end{subfigure}
    \caption{Condition number and corresponding number of GMRES iterations required to solve the linear systems produced by the Galerkin discretization of the EFIE and the MFIE. \textit{Left:} A cuboid of dimension $\mathrm{1m \times 1m \times 0.25m}$ with a concentric hole of dimension $\mathrm{0.5m \times 0.5m \times 0.25m}$. \textit{Right:} A pyramid of height $0.5 \mathrm{m}$ and 24-pointed star base, whose vertices lie on two circles of radius $1\mathrm{m}$ and $0.3\mathrm{m}$. Whereas the condition number and the corresponding GMRES iteration counts of the EFIE grow when the meshwidth $h$ decreases, those of the MFIE stay almost constant.}
    \label{fig:cond}
\end{figure}

\subsection{Star-based pyramid}
The second experiment concerns the scattering by a perfectly conducting pyramid with height $0.5\mathrm{m}$ and 24-pointed star base whose vertices lie on two circles of radius $1\mathrm{m}$ and $0.3\mathrm{m}$, see Figure~\ref{fig:domains} (\textit{right}).

On a polyhedral boundary $\Gm$, the regularity parameter $s^\ast$ can be determined by
\[
    s^\ast = \min_{v \in \Gm} \dfrac{2\pi}{\abs{\pa\omega_v}},
\]
where $\omega_v$ is the domain on the unit sphere cut out by the tangent cone to $\Gm$ at a vertex $v$ of $\Gm$ (see \cite[Theorem~8]{BCS2002a}). On the considered star-based pyramid, the boundary length of the curvilinear polygon $\omega_v$ reaches its maximum at the bottom vertex $v_0$ of the boundary $\Gm$, at which 48 edges meet. Particularly, we can find $s^\ast \approx 0.228$. Hence, according to Theorem~\ref{thm:convergence_rate}, only a sub-optimal convergence rate $\OO(h^{1.228})$ can be expected on this domain. Figure~\ref{fig:error} (\textit{right}) confirms our convergence analysis. Here, the reference solution $\ub_{\mathrm{ref}}$ is computed at the meshwidth $h = 0.02\mathrm{m}$. Finally, Figure~\ref{fig:cond} (\textit{right}) confirms again the well-conditioning property of the proposed Galerkin discretization scheme \eqref{eq:dis_vf} for the MFIE.

\section{Conclusions}
\label{sec:conclusions}

We have provided rigorous analysis of a Galerkin boundary element discretization of the magnetic field integral equation on Lipschitz polyhedra. The continuous variational problem has been shown to be uniquely solvable, provided that the wave number does not belong to the spectrum of the interior Maxwell's problem. A Petrov-Galerkin discretization of the variational formulation employing Raviart-Thomas boundary elements for the solution space and Buffa-Christiansen elements for the test space has been analyzed. Under a stability condition, the solvability and the well-conditioning of the discrete system have been proven, and the convergence rate of the discretization scheme has been investigated. A numerical scheme to verify the stability condition has been proposed, showing that this condition is unproblematic in practice. Some numerical results have also been presented, corroborating the theoretical analysis as well as the effect of the singularities of Maxwell's solutions on polyhedra to the convergence rate. The analyzed discretization scheme has been widely used in practical applications.

In future works, we shall investigate Galerkin discretization schemes for combined field integral equations and single source integral equations that employ the proposed discretization scheme for the double layer part.

\section*{Acknowledgments}

The authors would like to thank Dr. Ignace Bogaert for the many insightful comments and discussions.

\bibliographystyle{siamplain}
\bibliography{abrv_ref}
\end{document}